\documentclass[a4paper,10pt]{scrartcl}

\usepackage[utf8]{inputenc}
\usepackage[T1]{fontenc}
\usepackage[ngerman,english]{babel}
\usepackage{amsmath, amsthm, amssymb} 
\usepackage{mathtools}
\usepackage[usenames,dvipsnames]{color}
\usepackage{tikz}
\usepackage{fullpage}
\allowdisplaybreaks

\newcommand{\nn}{\nonumber}

\newcommand{\eps}{\varepsilon}
\newcommand{\N}{\mathbb{N}}
\newcommand{\R}{\mathbb{R}}

\newcommand{\nat}{\in\N}
\newcommand{\ny}{\nu}
\newcommand{\my}{\mu}
\newcommand{\phii}{\varphi}
\newcommand{\uep}{u_{\eps}}
\newcommand{\cep}{c_{\eps}}
\newcommand{\nep}{n_{\eps}}
\newcommand{\uept}{u_{\eps t}}
\newcommand{\cept}{c_{\eps t}}
\newcommand{\nept}{n_{\eps t}}

\newcommand{\cnep}{c_{0\eps}}

\newcommand{\Om}{\Omega}
\newcommand{\na}{\nabla}
\newcommand{\del}{\partial}
\newcommand{\Liom}{L^\infty(\Om)}
\newcommand{\intttpe}{\int_{t}^{t+1}}
\newcommand{\intnT}{\int_0^T}
\newcommand{\intninf}{\int_0^\infty}
\newcommand{\ntilde}{\widetilde{n}}
\newcommand{\ctilde}{\widetilde{c}}
\newcommand{\amrand}{|_{\rand}}
\newcommand{\rand}{\del\Omega}
\newcommand{\intttz}{\int_{t_0}^{t_0+2}\!}

\newcommand{\Laplace}{\Delta}
\newcommand{\Lap}{\Laplace}
\newcommand{\calF}{\mathcal{F}}
\newcommand{\calG}{\mathcal{G}}
\newcommand{\calP}{\mathcal{P}}
\newcommand{\cbar}{\overline{c}}
\newcommand{\Ombar}{\overline{\Om}}
\newcommand{\dom}{\del \Om}
\newcommand{\dOm}{\dom}
\newcommand{\bdry}{\amrand}
\newcommand{\intom}{\int_\Om}
\newcommand{\intnt}{\int_0^t}
\newcommand{\io}{\intom}
\newcommand{\intdom}{\int_{\dom}}

\newcommand{\Lom}[1]{L^{#1}(\Om)}
\newcommand{\ddt}{\frac{\mathrm{d}}{\mathrm{d}t}}
\newcommand{\set}[1]{\{#1\}}

\newcommand{\setl}[1]{\left\{#1\right\}}
\newcommand{\delny}{\partial_\ny}
\newcommand{\sub}{\subset}

\newcommand{\weakto}{\rightharpoonup}
\newcommand{\weakstarto}{\overset{*}{\weakto}}

\newcommand{\upto}{\nearrow}
\newcommand{\downto}{\searrow}
\newcommand{\embeddedinto}{\hookrightarrow}
\newcommand{\norm}[2][]{\left\|#2\right\|_{#1}}
\newcommand{\nnorm}[2]{\left\|#2\right\|_{#1}}

\newtheorem{theorem}{Theorem}
\numberwithin{theorem}{section}
\newtheorem{lemma}[theorem]{Lemma}
\newtheorem{corollary}[theorem]{Corollary}
\newtheorem{proposition}[theorem]{Proposition}

\newtheorem{definition}[theorem]{Definition}

\usepackage{units}
\usepackage{newunicodechar}
\newunicodechar{κ}{{\kappa}}
\newunicodechar{μ}{{\mu}}
\newunicodechar{α}{\alpha}
\newunicodechar{β}{\beta}
\newunicodechar{η}{\eta}
\newunicodechar{ϕ}{\phi}
\newunicodechar{φ}{\varphi}
\newunicodechar{Φ}{\Phi}
\newunicodechar{ℝ}{\mathbb{R}}
\newunicodechar{ℕ}{\mathbb{N}}
\newunicodechar{∞}{\infty}
\newunicodechar{ε}{\varepsilon}
\newunicodechar{∈}{\in}
\newunicodechar{γ}{\gamma}
\newunicodechar{δ}{\delta}
\newunicodechar{σ}{\sigma}
\newunicodechar{τ}{\tau}
\newunicodechar{ξ}{\xi}
\newunicodechar{•}{\cdot}
\newunicodechar{Δ}{\Delta}
\newunicodechar{∇}{\nabla}
\newunicodechar{Ω}{\Omega}
\newunicodechar{χ}{\chi}
\newunicodechar{Ψ}{\Psi}
\newunicodechar{∃}{\exists}
\newunicodechar{∀}{\forall}
\newunicodechar{ζ}{\zeta}

\newcommand{\epscond}{\ensuremath{\eps>0}}

\newcommand{\Yep}{Y_\eps}
\newcommand{\GNI}{Gagliardo-Nirenberg inequality}
\newcommand{\YI}{Young's inequality}
\newcommand{\onni}{\qquad \mbox{ on } (0,∞)}
\newcommand{\ontme}{\quad \mbox{ on } (0,\Tmaxe)}
\newcommand{\fate}{\onni \mbox{ and for all } ε>0}
\newcommand{\fat}{\qquad \mbox{for all }t>0}
\newcommand{\kl}[1]{\left(#1\right)}
\newcommand{\LT}[2]{L^{#1}((0,T),#2)}
\newcommand{\Tmaxe}{T_{max,ε}}
\newcommand{\epssys}{\eqref{epsys}}
\newcommand{\iNTme}{\in(0,\Tmaxe)}

\author{Johannes Lankeit}
\title{Long-term behaviour in a chemotaxis-fluid system with logistic source}
\date{}

\begin{document}
\maketitle 
\begin{abstract}
\noindent We consider the coupled chemotaxis Navier-Stokes model with logistic source terms
\begin{align*}
 n_t + u\cdot \nabla n &= \Delta n - \chi \nabla \cdot (n \nabla c)  + \kappa n - \mu n^2 \\
 c_t + u\cdot \nabla c &= \Delta c - nc \\
 u_t + (u\cdot \nabla)u &= \Delta u +\nabla P + n\nabla \Phi + f, \quad\qquad \nabla \cdot u=0 
\end{align*}
 in a bounded, smooth domain $\Omega\subset \mathbb{R}^3$ under homogeneous Neumann boundary conditions for $n$ and $c$ and homogeneous Dirichlet boundary conditions for $u$ and with given functions $f\in L^\infty(\Omega\times(0,\infty))$ satisfying certain decay conditions and $\Phi\in C^{1+\beta}(\bar\Omega)$ for some $\beta\in(0,1)$.\\
 We construct weak solutions and prove that after some waiting time they become smooth and finally converge to the semi-trivial steady state $(\frac{\kappa}{\mu},0,0)$.
 \textbf{Math Subject Classification (2010):}  35B40, 35K55, 35B65, 35Q30, 92C17
\end{abstract}

\section{Introduction}
Bacteria and sand are different. Although both are heavier than water and will tend to sink if dispersed in it, bacteria may possess the ability to swim -- and to direct their movement toward more favorable environmental conditions, i.e. for example toward higher concentration of oxygen, thus instigating the emergence of bioconvective patterns (see \cite[Sec. 4.2]{pedley_kessler}). Such behaviour can, e.g., be observed if colonies of \textit{Bacillus subtilis} are suspended in a drop of water (see e.g. \cite{kessler_burnett_remick_book}), and models describing this phenomenon, that is, the model proposed by Tuval et al. in \cite{tuval} and variants thereof, have received much attention from the mathematical community over the past few years. 

Before we recall some of the progress made in the analysis of such models, 
let us briefly motivate the form of the system we want to investigate in the present article. 
In order to describe the interaction between bacteria, their fluid environment and oxygen (or another nutrient) contained therein, we introduce scalar-valued functions $n$ and $c$ standing for the concentration of bacteria and oxygen, respectively, and 
a vector-valued function $u$ representing the velocity field of the surrounding water. 
The fluid motion is supposed to be governed by the incompressible Navier-Stokes equations
\[
 u_t + (u•∇)u = Δu + ∇P + n∇Φ + f, \quad\qquad \na\cdot u=0
\]
where we have allowed for an external force $f$ (which nevertheless might best be thought of as being zero in the most prototypical case) and, more importantly, where bouyancy effects are included, which arise from density differences between fluid with and without bacteria, as mandated by the presence of a given gravitational potential $\Phi$. 
$P$ symbolizes the pressure of the fluid, another unknown quantity. 

Oxygen is assumed to diffuse in the manner of linear diffusion, as described by the heat equation. It is moreover transported in the direction of the fluid flow and, finally, consumed with a rate proportional to the amount of bacteria present. 
Combining these effects, the resulting equation is the following: 
\[
 c_t=\underbrace{\Delta c}_{\mathrm{diffusion}} \underbrace{- nc}_{\mathrm{consumption}} \underbrace{- u\cdot \na c }_{\mathrm{transport}}.
\]

The evolution of the bacterial concentration is also influenced by diffusion and transport along the velocity field of the fluid. The cells moreover steer their motion in the direction of the concentration gradient of oxygen, by means of chemotaxis. This gives rise to a contribution $-χ\na \cdot (n\na c)$ to the time derivative of $n$, thus introducing cross-diffusive effects into the model, which lie at the core of the mathematical difficulties accompanying the analysis of chemotaxis systems like the famous Keller-Segel model (\cite{horstmann_I,BBTW}). Therein $χ>0$ is a parameter regulating the strength of the chemotactic attraction.
In addition, we want to allow for population growth to take place in the simplest conceivable manner, namely according to a logistic law, where we denote by $κ$ the effective growth rate of the population and by $μ$ a parameter controlling death by overcrowding. 
In total, these effects yield the equation:
\[
 n_t = \underbrace{Δ n}_{\mathrm{diffusion}} \underbrace{- χ∇•(n ∇ c)}_{\mathrm{chemotaxis}} \underbrace{- u•∇n}_{\mathrm{transport}} \underbrace{+κn-μn^2}_{\mathrm{logistic\ growth}}
\]
With time starting at $0$, spatially the whole scenario is to take place in a bounded domain $Ω\subset ℝ^3$ with smooth boundary, which we want to think of as drop of water resting on a surface. Thus it is quite natural to assume that no fluid motion takes place on the surface of the drop, that is
\[
 u = 0 \qquad \text{on } \dOm,
\]
and that no bacteria cross the boundary between the drop and its surroundings
\[
 \delny n =0 \qquad \text{on } \dOm.
\]
We will also assume that 
\[
 \delny c = 0 \qquad \text{on } \dOm,
\]
that is, that no exchange of oxygen takes place between the fluid environment and its exterior. This assumption is less natural, at least for the part of the boundary that separates water and air, but so far has been employed in almost all papers dealing with chemotaxis fluid interaction from a mathematical viewpoint 
(exceptions being 
early existence results for weak solutions in 2-dimensional bounded domains \cite{Lorz}, 
numerical experiments like in \cite{chertock_etal_numeric}
and, most notably, a recent work by Braukhoff \cite{braukhoff}, where it was shown that in 2- or 3-dimensional convex bounded domains classical or weak solutions, respectively, exist for a chemotaxis-Navier-Stokes model with logistic source if the boundary condition for $c$ is $\delny c=1-c$). 

Thus, in total the system to be considered here is 
\begin{align}\label{eq:system}
 n_t + u•∇n &= Δn - χ∇•(n∇c)  + κn - μn^2 & & \mbox{in } \Om\times(0,∞)\nn\\
 c_t + u•∇c &= Δ c - nc & & \mbox{in } \Om\times(0,∞)\nn\\
 u_t + (u•∇)u &= Δ u +∇P + n∇Φ + f, \quad\qquad ∇•u=0& & \mbox{in } \Om\times(0,∞)\\
 u=0,\quad  \delny n&=\delny c=0\nn& & \mbox{in } \dOm\times(0,∞)\\
 n(\cdot,0)=n_0,\quad c(\cdot,0)&=c_0,\quad u(\cdot,0)=u_0\nn& & \mbox{in } \Om
\end{align}
for some initial data 
\begin{align}\label{eq:init}
 \begin{cases}n_0 >0,\qquad c_0>0\qquad \mbox{in } \Ombar &\\
 n_0\in C^0(\Ombar),\quad c_0\in W^{1,q}(\Om),\quad u_0\in D(A^\alpha)&\end{cases}
\end{align}
with $q>3$, $\alpha\in(\frac34,1)$, where $A$ denotes the realization of the Stokes operator under homogeneous Dirichlet boundary conditions in the solenoidal subspace $L^2_\sigma(\Om)$ of $L^2(\Om)$.

If $κ=μ=0$ (and $f\equiv 0$), this model is an instance of the one for which the existence of global weak solutions in $\Om=ℝ^2$ was shown in \cite{Liu_Lorz}. The existence of global classical solutions in two-dimensional bounded convex domains was discovered in \cite{wk_ctfluid}. Global weak solutions on $\Om=\R^2$ have been treated in \cite{zhang_zheng} under weaker conditions on the initial data. In the setting of \cite{wk_ctfluid}, the convergence of solutions to the stationary state was proven in \cite{wk_arma}; its rate was given in \cite{zhang_li_decay}. Upon neglection of the nonlinear fluid term $(u\cdot∇)u$, that is upon consideration of Stokes flow instead of a Navier-Stokes governed fluid, global weak solutions can also be found in bounded three-dimensional domains (\cite{wk_ctfluid}).
(The results of \cite{wk_ctfluid,wk_arma} have been extended to non-convex domains in \cite{jiang_wu_zheng}.) For the three-dimensional setting (of bounded convex domains) with full Navier-Stokes-fluid and large initial data only recently the existence of weak solutions has been demonstrated by Winkler (\cite{wk_ctfluid3dnastoexist}). He furthermore showed that any eventual energy solution becomes smooth after some waiting time, and converges as $t\to\infty$ (\cite{wk_ns_oxytaxis}).

Other variants of the model that are commonly treated include nonlinear (porous medium type) diffusion of bacteria, where $\Delta n$ is replaced by $\Delta n^m$ for some $m>1$ (see \cite{TaoWk_KS-stokes-porous,TaoWk_ctS3d_nonlin,duan_xiang,chung_kang_kim,zhang_li}), thereby improving chances for finding bounded solutions, 
or, exchanging $χ∇•(n∇c)$ for $∇•(nS(n,c,x)∇c)$, more complex sensitivity functions $S$ (\cite{wk_ctfluid3dnonlineargeneral,wang_xiang,xinru_yulan,sachiko_positiondep,XinruJoh}), 
which may be matrix-valued, thus introducing new mathematical challenges by destroying the natural energy structure of the system and, seen from the biological viewpoint, taking care of more complicated swimming behaviour of bacteria (cf. \cite{ecoli_RHS,ramia_tullok_phanthien,xue_othmer}).

In contrast to \eqref{eq:system}, in the classical Keller-Segel system the chemoattractant is produced by the bacteria themselves and not consumed (accounting for terms $+n-c$ in place of $-nc$ in the second equation of \eqref{eq:system}), and models of Keller-Segel-Stokes type have also been considered (\cite{wang_xiang,tobias_stokes_sublinear}). In $\Om=ℝ^3$, mild solutions to a system encompassing both mechanisms at the same time were proven to exist under a smallness condition on initial data (\cite{kozono}).

Chemotaxis fluid models including logistic growth ($κ,μ>0$) have been treated in \cite{vorotnikov,taowk_quadrdegradation,taowk_quadrdegradation3d,wk_ksns_logsource,braukhoff}. 

In \cite{vorotnikov}, a result on the existence of weak solutions for \eqref{eq:system} is given, and for the case of sufficiently nonlinear cell diffusion, attractors are considered. In a Keller-Segel-Navier-Stokes system with logistic source ($μ>0$, $κ\ge 0$) in two-dimensional bounded domains global classical solutions have been detected in \cite{taowk_quadrdegradation}, which furthermore converge to $0$ if $κ=0$. 
Under the assumptions of a Stokes fluid and sufficiently large $μ$ (explicitly: $μ>23$), in \cite{taowk_quadrdegradation3d} these results have been achieved for three-dimensional bounded domains as well. 
In \cite{wk_ksns_logsource}, for $μ>\frac{1}4 \sqrt{κ_+}χ$ in bounded convex domains $\Om\subℝ^3$ generalized solutions are constructed, which then are shown to converge to the homogeneous steady state with respect to the topology of $\Lom 1\times\Lom p\times\Lom 2$ for $p\in[1,6)$, if certain conditions on $f$ are satisfied.

It is the main goal of the present article to achieve similar results for the consumption-chemotaxis-fluid model \eqref{eq:system}. 
Having to deal with a consumption instead of production term in the $c$-equation seems more beneficial for proving boundedness of solutions and encourages us to hope that the solutions remain bounded and thus exist globally without any further largeness condition on $μ$ except positivity and that the convergence takes place with respect to stronger topologies than in \cite{wk_ksns_logsource}. This is indeed what we will prove. Moreover, we will shed light on asymptotic regularity properties of the solutions we are going to construct. 

Let us state the main results in detail: 
Posing the condition 
\begin{equation}\label{reg:f}
\begin{cases} f\in L^2((0,\infty),L^{\frac 65}(\Om)) \cap L^\infty(\Om\times(0,\infty))\cap C^{β,\frac{β}2}(\Ombar\times[0,\infty)),\\
 \norm[\Lom {\frac32}]{f(\cdot,t)}\to 0 \quad \mbox{as }t \to \infty \qquad \mbox{for some } α>0 \end{cases}
\end{equation}
on the external force on the fluid, we will first (re-)derive the following theorem on global existence of weak solutions: 

\begin{theorem}\label{thm:exweaksol}
Let $\Om\subℝ^3$ be a bounded smooth domain and let $χ,κ\ge 0$, $μ>0$.  
Let $n_0, c_0, u_0$ be as in \eqref{eq:init} with some $q>3$ and $α\in(\frac34,1)$, let $Φ\in C^{1+β}(\Ombar)$ for some $β>0$, and let $f$ satisfy \eqref{reg:f} for some $β\in(0,1)$. 
Then there is a weak solution (in the sense of Definition \ref{def:weaksol} below) to \eqref{eq:system}, which can be approximated by a sequence of solutions $(\nep,\cep,\uep)$ to \eqref{epsys} in a pointwise manner (and moreover with respect to the topologies indicated in Proposition \ref{prop:ex-weaksol}).
\end{theorem}

(For weak solutions to \eqref{eq:system} with $f\equiv 0$ see also \cite[Thm. 4.1]{vorotnikov} or, for a setting with different boundary conditions, \cite{braukhoff}.) 
The solutions $(\nep,\cep,\uep)$ to the approximate system \eqref{epsys} that are mentioned in Theorem \ref{thm:exweaksol} (but do not appear in \cite{vorotnikov}) will serve as essential tool also in the proof of our second theorem, which is concerned with the asymptotic behaviour and eventual regularity of solutions.

\begin{theorem}\label{thm:evsmooth}
Let the assumptions of Theorem \ref{thm:exweaksol} be satisfied. Then there are $T>0$ and $γ\in(0,1)$ such that the solution $(n,c,u)$ given by Theorem \ref{thm:exweaksol} satisfies 
\[
 n,c \in C^{2+γ,1+\frac{γ}2}(\Ombar\times[T,\infty)),\quad u\in C^{2+γ,1+\frac{γ}2}(\Ombar\times[T,\infty)).
\]
Moreover, 
\[
 n(\cdot,t)\to \frac{κ}{μ}, \quad c(\cdot,t)\to 0,\quad u(\cdot,t)\to 0\qquad \mbox{as } t\to∞, 
\]
where the convergence takes place with respect to the norm of $C^1(\Ombar)$.
\end{theorem}

As to the proofs, we will first turn our attention to Theorem \ref{thm:exweaksol}: In Section \ref{sec:exweak}, namely, we will be concerned with solutions to the approximate problem \epssys\ (see Lemma \ref{lem:locex}) and with the derivation of estimates that allow for compactness arguments in constructing solutions to \eqref{eq:system} (Proposition \ref{prop:ex-weaksol}).
The foundation for the acquisition of these estimates will be an examination of the derivative of 
\[
  \io\nep\ln\nep +\frac{\chi}2 \io \frac{|\na \cep|^2}{\cep} + K\chi\io |\uep|^2
\]
for suitable $K>0$ (see Lemma \ref{lem:ddtF}). 
In contrast to a system without logistic source terms in the equation for $n$, mass conservation of the bacteria is not guaranteed in \eqref{eq:system}.
We begin Section \ref{sec:thm2} by finding a suitable substitute, and then, relying on this, prove convergence of $\int_t^{t+1} \io \cep$ and of $\norm[\Lom{∞}]{\cep(•,t)}$ as $t\to\infty$ (Lemma \ref{lem:iintcto0} and Lemma \ref{lem:clinto0}, respectively).
In Lemma \ref{lem:nepbd} we derive a differential inequality for $\io \frac{\nep^p}{(\eta-\cep)^\theta}$ for appropriate parameters $η, \theta$, finally yielding $L^p$-bounds on $n$ whenever the second solution component is small. 
Using these bounds, we then prove eventual Hölder regularity of $\uep$ (Lemma \ref{lem:uepc1alpha}), $\cep$ (Lemma \ref{lem:cc1pa}), and $\nep$ (Lemma \ref{lem:nc1pa}), which can be transferred to $n,c,u$ and turned into higher regularity (Lemma \ref{lem:c2alpha}).
For convergence as $t\to \infty$, we finally draw upon uniform Hölder bounds (Corollary \ref{cor:convc1pa}) and the compact embedding $C^{1+α,\frac{α}2}(\Ombar\times[t,t+1])\embeddedinto C^{1,0}(\Ombar\times[t,t+1])$ as well as some of the properties collected during the course of Section \ref{sec:thm2}; concerning $n$, for example, Lemma \ref{lem:estimatesfromddtG} (and thus, indirectly, Lemma \ref{lem:nleq}) will once more be important.

\textbf{Notation.} Given any function $w$ defined on $\Omega\times[0,T)$ for some $T\in(0,\infty]$, we define $w(t):=w(\cdot,t)$ for any $t\in[0,T)$. We will refer to the partial derivative with respect to the last argument by $\frac{d}{dt} w$. The symbol $\embeddedinto\embeddedinto$ will be used to indicate compact embeddings. For vectors $v,w\inℝ^3$ we let $v\otimes w$ denote the matrix $(v_iw_j)_{i,j=1,2,3}$. $\calP\colon L^p(\Om)\to L^p_\sigma(\Om)$ stands for the Helmholtz projection in $\Lom p$.

\section{Existence of weak solutions}\label{sec:exweak} 

We will start by considering an approximate problem, namely
\begin{subequations}\label{epsys}
\begin{align}
 \nept+\uep•\na\nep=&\Lap\nep-\chi\na\cdot\left(\frac\nep{1+\eps\nep}\na\cep\right)+\kappa\nep-\my\nep^2 \label{eq:nep}\\
  \cept+\uep•\na\cep=&\Lap\cep-\cep\frac1\eps\ln(1+\eps\nep) \label{eq:cep}\\
  \uept+(Y_\eps \uep•\na)\uep=&\Lap\uep+\na P_\eps +\nep\na\Phi+f(x,t) \label{eq:uep}\\
  \delny \nep\bdry = \delny \cep\bdry=&0, \qquad \uep\bdry=0\\
  \nep(\cdot,0) =n_0, \quad \cep(\cdot,0)=&c_0, \quad \uep(\cdot,0)=u_0
 \end{align}
\end{subequations}
where $Y_\eps=(1+\eps A)^{-1}$, and provide estimates for its solutions. In Proposition \ref{prop:ex-weaksol}, these estimates will enable us to construct a solution to \eqref{eq:system} by a limiting process. 
An approximation in this way was also employed in \cite{wk_ctfluid3dnastoexist}, 
\cite{wk_ns_oxytaxis}, 
 \cite{wk_ksns_logsource}.

\subsection{Local existence and basic properties}
First, let us recall that locally these solutions actually exist. 
Because the reasoning is well-established (and not central to later parts of the article), 
we shall only briefly hint at the proofs, both here and in Lemma \ref{lem:epglobal}, where their global existence is indicated. 
\begin{lemma}\label{lem:locex}
 Let $q>3$, $\alpha \in (\frac34,1)$, $β\in(0,1)$, $κ,χ\geq0$, $μ>0$, $Φ\in C^{1+β}(\Ombar)$, $f$ as in \eqref{reg:f}, let $n_0, c_0, u_0$ satisfy \eqref{eq:init} and let $\eps>0$. Then there are $\Tmaxe$ and uniquely determined functions 
\begin{align*}
 \nep&\in C^0(\Ombar\times[0,\Tmaxe))\cap C^{2,1}(\Ombar\times(0,\Tmaxe)),\\
 \cep&\in C^0(\Ombar\times[0,\Tmaxe))\cap C^{2,1}(\Ombar\times(0,\Tmaxe))\cap L^\infty((0,\Tmaxe),W^{1,q}(\Om)),\\
 \uep&\in C^0(\Ombar\times[0,\Tmaxe))\cap C^{2,1}(\Ombar\times(0,\Tmaxe))
\end{align*}
which together with some $P_\eps\in C^{1,0}(\Ombar\times(0,\Tmaxe))$ solve \epssys\ classically, and satisfy $\Tmaxe=\infty$ or 
\begin{equation}\label{eq:continuationcrit}
 \limsup_{t\upto\Tmaxe} \norm[\Lom\infty]{\nep(\cdot,t)}+\norm[W^{1,q}(\Om)]{\cep(\cdot,t)}+\norm[\Lom2]{A^\alpha\uep(\cdot,t)}=\infty.
\end{equation}
\end{lemma}
\begin{proof}
 The proof follows the reasoning of \cite[Lemma 2.1]{wk_ctfluid3dnastoexist} if some of the adaptions necessary in \cite[Lemma 2.2]{wk_ctfluid3dnastoexist} and \cite[Lemma 3.1]{wk_ksns_logsource} are taken into account.

 Banach's fixed point theorem applied in closed ball in $\LT{\infty}{C^0(\Ombar)\times W^{1,q}(\Om)\times D(A^\alpha)}$ to a function whose fixed points are mild solutions to the system establishes the existence of such solutions on a time interval $[0,T)$, where $T$ depends on the norms featured in \eqref{eq:continuationcrit} only. By an invocation of standard regularity theory for parabolic equations and the Stokes semigroup these solutions turn into classical solutions. 
\end{proof}

For the rest of the article 
 let us fix parameters $χ,κ\geq 0$, $μ>0$, $α\in(\frac34,1)$, $q>3$, $β\in(0,1)$, $f$ as in \eqref{reg:f}, $\Phi\in C^{1+β}(\Ombar)$, initial data $n_0,c_0,u_0$ satisfying \eqref{eq:init} and, given \epscond, let us denote by $(\nep,\cep,\uep)$ the corresponding solution to \eqref{epsys}.

\begin{lemma} For any $x\in\Om$, \epscond, $t\in(0,\Tmaxe)$ we have $\nep(x,t)\geq0$ and $\cep(x,t)\geq 0$.
\end{lemma}
\begin{proof}
 An application of the parabolic comparison principle to the subsolution $0$ of \eqref{eq:nep} or \eqref{eq:cep}, respectively, immediately results in the claimed nonnegativity.
\end{proof}

Similarly, we obtain boundedness of $\cep$.

\begin{lemma}\label{lem:cbd}
 There is $C>0$ such that for any \epscond\ and for any $t>\iNTme$, 
\begin{equation}
 \norm[\Liom]{\cep(t)}\leq C 
\end{equation}
 and, for any $\eps>0$, $t\mapsto \norm[\Liom]{\cep(t)}$ is nonincreasing on $(0,\infty)$.
\end{lemma}
\begin{proof}
 With $C=\norm[\Liom]{c_0}$, both assertions are a consequence of the parabolic comparison principle.
\end{proof}

Another quantity whose boundedness, in this case in $\Lom 2$, quickly results from the second equation is the gradient of $\cep$:
\begin{lemma}\label{lem:intnac2}
 There is $C>0$ such that for any \epscond 
\begin{equation}\label{eq:intnac}
 \int_0^{\Tmaxe}\!\!\io |\na\cep|^2\leq C.
\end{equation}
\end{lemma}
\begin{proof}
 Let $ε>0$. 
 Upon multiplication by $\cep$, integration over $\Om$ and integration by parts, \eqref{eq:cep} results in 
\[
 \io \cep\cept = -\io |\na\cep|^2 -\io \cep^2\frac1\eps\ln(1+\eps\nep) + \io \cep \uep•\na\cep \ontme, 
\]
where the last term vanishes by $\io \cep \uep•\na\cep=\frac12\io\uep•\na(\cep^2)=-\frac12\io\cep^2\na\cdot\uep=0$ due to $\na\cdot \uep(t)=0$ for all $t\iNTme$, and integration with respect to time entails 
\[
 \frac12\io \cep^2(t) +\intnt \io |\na\cep|^2\leq \frac12\io \cnep^2\qquad \mbox{for any }t\iNTme,
\]
so that we may conclude \eqref{eq:intnac} by taking $t\upto\Tmaxe$. 
\end{proof}

In contrast to the situation without source terms, we cannot hope for mass conservation in the first component. Nevertheless, the following inequality still holds:
\begin{lemma}\label{lem:nleq}
 There is $C>0$ such that for any \epscond 
 \[
  \io \nep(t) \leq C \  \mbox{ for all } t\iNTme \quad \mbox{ and }\quad \int_t^{t+τ}\!\!\io \nep^2 \leq C \  \mbox{ for any } t\in(0,\Tmaxe-τ), 
 \]
 where $τ:=\min\set{1,\frac12\Tmaxe}$.
\end{lemma}
\begin{proof}
 Integration of \eqref{eq:nep} and an application of H\"older's inequality yield that for any $ε>0$
\[
 \ddt \io \nep = \kappa\io\nep -\my\io\nep^2\leq \kappa\io \nep - \frac{\my}{|\Om|} \left(\io\nep\right)^2 \ontme,
\]
so that an ODE comparison argument gives boundedness of $\io\nep(t)$ and integration with respect to time allows to conclude the existence of a bound on $\int_t^{t+τ}\!\!\!\io \nep^2$ by means of $\my\int_t^{t+τ}\!\!\!\io\nep^2=\kappa\int_t^{t+τ}\!\!\!\io\nep+\io \nep(t)-\io\nep(t+τ)$.
\end{proof}

\subsection{A priori estimates implied by an energy type inequality}
We want to derive a (quasi-)energy inequality for the function
\begin{equation}\label{eq:defF}
  \io\nep\ln\nep +\frac{\chi}2 \io \frac{|\na \cep|^2}{\cep} + K\chi\io |\uep|^2.
\end{equation}
As preparation, we first deal with the derivatives of the summands separately: 

\begin{lemma}\label{lem:ddtnlnn}
 There is $C>0$ such that for any \epscond
\[
 \ddt \io \nep\ln\nep + \frac\my2 \io \nep^2\ln\nep + \io \frac{|\na\nep|^2}{\nep} \leq  \chi\io \frac{\na \nep•\na\cep}{1+\eps\nep} + C \ontme.
\]
\end{lemma}
\begin{proof}
 First we observe that $s\mapsto \kappa s- \my s^2, s\in[0,\infty)$, and $s\mapsto (\kappa s-\frac\my2 s^2)\ln s, s\in (0,\infty)$, are bounded from above by some constant $C_1$. Using these estimates and \eqref{eq:nep}, from integration by parts we obtain
\begin{align*}
 \ddt \io \nep\ln\nep =& \io \nept\ln\nep+\io\nept\\
&=\io \Lap \nep(\ln\nep)-\chi\io \ln\nep \na\cdot\kl{\frac{\nep}{1+\eps\nep}\na \cep} - \io \uep•\na\nep\ln\nep \\
&\qquad + \kappa\io \nep\ln\nep-\my\io \nep^2\ln\nep + \kappa\io\nep-\my\io \nep^2\\
&\leq -\io \frac{|\na \nep|^2}{n} + \chi \io \frac{\na \nep•\na\cep}{1+\eps n} - \frac\my2 \io \nep^2\ln\nep + 2C_1 \ontme
\end{align*}
for any $ε>0$, so that the claim results with $C=2C_1$.
\end{proof}

In the next lemma we will collect statements that will enable us to deal with terms arising from differentiation of $\io|\na\sqrt{\cep}|^2$. In particular, it is this lemma that will render any convexity condition on the domain unnecessary. 
The proofs are either contained in or adapted from the articles \cite{mizoguchi_souplet,sachikosekitomomi,wk_ctfluid}.

\begin{lemma}\label{lem:bdryterm}
 i) Let $w\in C^2(\Ombar)$ satisfy $\delny w=0$ on $\dOm$. Then 
 \[
  \delny |\na w|^2 \leq \mathfrak{K} |\na w|^2,
 \]
 where $\mathfrak{K}$ is an upper bound on the curvature of $\dOm$.\\
 ii) Furthermore, for any $\eta>0$ there is $C(η)>0$ such that every $w\in C^2(\Ombar)$ with $\delny w=0$ on $\dOm$ fulfils
\[ \norm[L^2(\dOm)]{\na w}\leq \eta \norm[L^2(\Om)]{\Delta w} + C(\eta)\norm[L^2(\Om)]{w}.\]
 iii) For any positive $w\in C^2(\Ombar)$ 
 \begin{equation}\label{eq:Deltasqrt}
  \norm[L^2(\Om)]{\Delta \sqrt{w}} \leq \frac12\norm[L^2(\Om)]{\sqrt{w}\Delta \log w} + \frac14 \norm[L^2(\Om)]{w^{-\frac32}|\na w|^2}.
 \end{equation}
 iv) For any positive $w\in C^2(\Ombar)$ with $\delny w=0$ on $\dOm$ we have 
 \[
  -2\io \frac{|\Lap w|^2}{w} +\io \frac{|\na w|^2\Lap w}{w^2} = -2\io w|D^2\log w|^2 + \intdom \frac1w\delny|\na w|^2.
 \]
 v) There is $k>0$ such that for all positive $w\in C^2(\Ombar)$ with $\delny w=0$ on $\dOm$ the inequality 
 \[
  \io w|D^2\ln w|^2\geq k\io\frac{|\na w|^4}{w^3}
 \]
holds.\\
vi) There are $C>0$ and $k>0$ such that every positive $w\in C^2(\Ombar)$ fulfilling $\delny w=0$ on $\dOm$ satisfies
\begin{equation}\label{eq:lembdrytermvi}
  -2\io \frac{|\Lap w|^2}{w} +\io \frac{|\na w|^2\Lap w}{w^2} \leq -k \io w|D^2\ln w|^2 -k \io\frac{|\na w|^4}{w^3} + C\io w.
\end{equation}
\end{lemma}
\begin{proof}
 i) This is \cite[Lemma 4.10]{mizoguchi_souplet}.\\
 ii)  Let us fix $r\in(0,\frac12)$. Thanks to the boundedness of the trace operator $tr\colon W^{r+\frac12,2}(\Om)\to W^{r,2}(\dOm)$ (cf. \cite[Thm. 4.24]{TriebelHaroske}) and the embedding $W^{r,2}(\dOm)\embeddedinto L^2(\dOm)$ (see \cite[Prop. 4.22(ii)]{TriebelHaroske}) there is $k_1>0$ such that $\norm[L^2(\dOm)]{\psi}\leq k_1\norm[W^{r+\frac12,2}(\Om)]{\psi}$ for all $\psi\in W^{r+\frac12,2}(\Om)$. If we let $\theta:=\frac34+\frac r2$, the interpolation inequality \cite[Prop. 2.3]{LionsMagenesI} 
 guarantees the existence of $k_2>0$ such that $\norm[W^{r+\frac32,2}(\Om)]{\psi} \leq k_2 \norm[W^{2,2}(\Om)]{\psi}^{\theta}\norm[L^2(\Om)]{\psi}^{1-\theta}$ for all $\psi\in W^{2,2}(\Om)$. Furthermore, according to e.g. \cite[Thm. 19.1]{Friedman}, there is $k_3>0$ such that $\norm[W^{2,2}(\Om)]{\psi}\leq k_3\norm[L^2(\Om)]{\Delta \psi}+k_3\norm[L^2(\Om)]{\psi}$ for all $\psi\in W^{2,2}(\Om)$ with $\delny \psi=0$ on $\dOm$. Moreover, for any $\eta>0$, Young's inequality provides us with $k_4=k_4(\eta)$ such that for any $a,b\in[0,\infty)$ we have $a^\theta b^{1-\theta}\leq \frac{\eta}{k_1k_2k_3} a + k_4(\eta) b$, since the choice of $r$ implies $\theta\in(0,1)$.
 With these constants, for any $w\in W^{2,2}(\Om)$ satisfying $\delny w=0$ on $\dOm$ we obtain 
\begin{align*}
 \norm[L^2(\dOm)]{\na w} \leq& k_1\norm[W^{r+\frac12,2}(\Om)]{\na w}
 \leq k_1\norm[W^{r+\frac32,2}(\Om)]{w}
 \leq k_1k_2 \norm[W^{2,2}(\Om)]{w}^\theta \norm[L^2(\Om)]{w}^{1-\theta} \\
 \leq& k_1k_2k_3 \norm[L^2(\Om)]{\Delta w}^\theta\norm[L^2(\Om)]{w}^{1-\theta} + k_1k_2k_3\norm[L^2(\Om)]{w}^{\theta+1-\theta}
 \leq \eta\norm[L^2(\Om)]{\Delta w}+ (k_1k_2k_3+k_4(\eta))\norm[L^2(\Om)]{w}. 
\end{align*}
iii) Let $w\in C^2(\Ombar)$ be positive. The pointwise equalities 
 \[
  \Delta \sqrt{w}= \na\cdot\left(\frac{1}{2\sqrt{w}}\na w\right) = \frac{\Delta w}{2\sqrt{w}}+\frac12\na w\cdot\na\left(w^{-\frac12}\right)=\frac{\Delta w}{2\sqrt{w}}-\frac{|\na w|^2}{4w^{\frac32}}
 \]
 and 
 \[
  \Delta\log w=\na\cdot(\na\log w)=\na\cdot\left(\frac{\na w}{w}\right)=\frac{\Delta w}{w}-\frac{|\na w|^2}{w^2}
 \]
 immediately entail 
 \[
   \Delta\sqrt{w} = \frac12\sqrt{w}\Delta\log w+ \frac14\frac{|\na w|^2}{w^{\frac32}}
 \]
 and thus \eqref{eq:Deltasqrt}.\\
 iv) Being a special case of assertions from \cite{dPGG}, this can be found as Lemma 3.4 i) in \cite{singular_sensitivity}.\\
 v) This was proven as \cite[Lemma 3.3]{wk_ctfluid}.\\
 vi) Let $\eta>0$. Part i) and Young's inequality in combination with ii) and iii), respectively, can be employed to yield $C>0$ such that 
\begin{align*}
 \intdom \frac1w\delny|\na w|^2 \leq& \mathfrak{K}\intdom \frac{|\na w|^2}{w} = 4\mathfrak{K}\norm[L^2(\dOm)]{\na\sqrt{w}}^2 
\leq \eta\norm[L^2(\Om)]{\Delta\sqrt{w}}^2+C\norm[L^2(\Om)]{\sqrt{w}}^2\\
 \leq& \eta\norm[L^2(\Om)]{\sqrt{w}\Delta \log w}^2 + \eta \norm[L^2(\Om)]{w^{-\frac32}|\na w|^2}^2 + C \norm[L^2(\Om)]{\sqrt{w}}^2
\end{align*}
 for all positive $w\in C^2(\Ombar)$ satisfying $\delny w\amrand=0$.
 Thus, for any such $w$, 
 \[
  \intdom \frac1w\delny|\na w|^2 \leq \eta\io w|\Delta \log w|^2 + \eta \io \frac{|\na w|^4}{w^3} + C \io w.
 \]
 Taking into account iv) and v), we readily obtain \eqref{eq:lembdrytermvi}.
\end{proof}

We can take these estimates to their use in the next proof, which is concerned with the derivatives of the second summand in \eqref{eq:defF}.

\begin{lemma}\label{lem:ddtnacqc}
 There are $K,C,k>0$ such that for every \epscond
 \[
 \ddt \io \frac{|\na \cep|^2}\cep   
+k\io \cep|D^2\ln\cep|^2+k\io\frac{|\na\cep|^4}{\cep^3}
\leq C + K \io|\na \uep|^2 -2\io\frac{\na\cep\cdot\na\nep}{1+\eps\nep} \ontme.  
 \]
\end{lemma}
\begin{proof}
We begin by computing $\ddt \io \frac{|\na \cep|^2}\cep$: For any $\eps>0$ on $(0,\Tmaxe)$ we have 
\begin{align}\label{eq:ddtnacqc1}
 \ddt \io \frac{|\na \cep|^2}\cep =& 2\io \frac{\na \cep•\na \cept}{\cep}-\io\frac{|\na\cep|^2}{\cep^2}\cept\nn\\
 =&-2\io\frac{\Lap\cep\cept}{\cep}+2\io\frac{|\na\cep|^2}{\cep^2}\cept-\io\frac{|\na\cep|^2}{\cep^2}\cept\nn\\
 =&-2\io\frac{\Lap\cep\cept}{\cep}+\io\frac{|\na\cep|^2}{\cep^2}\cept\nn\\
 =&-2\io\frac{|\Lap\cep|^2}{\cep}+2\io\frac{\Lap\cep\frac1\eps\cep\ln(1+\eps \nep)}{\cep} +2\io\frac{\Lap\cep}{\cep}\uep\cdot\na\cep\nn\\
 &+\io\frac{|\na\cep|^2}{\cep^2}\Lap\cep-2\io\frac{|\na\cep|^2}{\cep^2}\cep\frac1\eps \ln(1+\eps \nep)-\io\frac{|\na\cep|^2}{\cep^2}\uep•\na\cep.
\end{align}
From Lemma \ref{lem:bdryterm} vi), we obtain $k_1>0, k_2>0$ such that for any $ε>0$ we may estimate 
\begin{equation}\label{eq:2.8est1}
 -2\io\frac{|\Lap\cep|^2}{\cep}+\io\frac{|\na\cep|^2}{\cep^2}\Lap\cep\leq -k_1\io \cep|D^2\ln\cep|^2-k_1\io\frac{|\na \cep|^4}{\cep^3} + k_2\io \cep \ontme.
\end{equation}
As to the terms containing $\uep$, we note that for all $ε>0$
\begin{align*}
 2\io \frac{\Lap \cep}\cep(\uep\cdot\na\cep) = 2\io\frac{|\na \cep|^2}{\cep^2} \uep\cdot\na\cep - 2\io \frac1{\cep}\na\cep•(\na\uep\na\cep)-2\io\frac{1}{\cep} \uep• D^2\cep\na\cep, 
\end{align*}
and 
\begin{align*}
 \io\frac{|\na\cep|^2}{\cep^2}\uep\cdot\na\cep=-\io\na\left(\frac1{\cep}\right)\cdot \uep|\na \cep|^2 = \io \frac{1}{\cep} \uep\cdot\na|\na\cep|^2=2\io \frac1{\cep}\uep• D^2\cep\na\cep 
\end{align*}
hold on $(0,\Tmaxe)$, 
so that for any $ε>0$ 
\begin{align}
 2\io\frac{\Lap\cep}{\cep}\uep\cdot\na\cep -\io\frac{|\na\cep|^2}{\cep^2}\uep•\na\cep =& \io\frac{|\na \cep|^2}{\cep^2} \uep\cdot\na\cep - 2\io \frac1{\cep}\na\cep•(\na\uep\na\cep)-2\io\frac{1}{\cep} \uep •D^2\cep\na\cep\nn\\
 =&- 2\io \frac1{\cep}\na\cep•(\na\uep\na\cep) \ontme,\label{eq:2.8ue1}
\end{align}
where we can estimate 
\begin{equation}\label{eq:2.8ue2}
 2\io \frac1{\cep}\na\cep•(\na\uep\na\cep)\leq \frac{k_1}2 \io \frac{|\na\cep|^4}{\cep^3} + k_3 \io \cep |\na \uep|^2 \ontme
\end{equation}
with some $k_3>0$ courtesy of Young's inequality.  
Moreover 
\begin{equation}\label{eq:2.8cn}
-2\io\frac{|\na\cep|^2}{\cep^2}\cep\frac1\eps \ln(1+\eps \nep)\leq 0 \qquad \mbox{ on } (0,\Tmaxe) \mbox{ and for all } ε>0
\end{equation}
and an integration by parts shows
\begin{equation}\label{eq:2.8clnn}
 2\io\frac{\Lap\cep\frac1\eps\cep\ln(1+\eps \nep)}{\cep}=-2\io\frac{\na\cep\cdot\na\nep}{1+\eps\nep}\fate
\end{equation}
so that, for any $ε>0$, using \eqref{eq:2.8est1}, \eqref{eq:2.8ue1}, \eqref{eq:2.8ue2}, \eqref{eq:2.8cn}, \eqref{eq:2.8clnn} turns  \eqref{eq:ddtnacqc1} into 
\[
 \ddt \io \frac{|\na \cep|^2}\cep  
+k_1\io \cep|D^2\ln\cep|^2+\frac{k_1}2\io\frac{|\na\cep|^4}{\cep^3}
\leq k_2\io \cep+ k_3\io \cep|\na\uep|^2 -2\io\frac{\na\cep\cdot\na\nep}{1+\eps\nep}
\]
on $(0,\Tmaxe)$ and finally inserting the uniform bound on $\cep$ provided by Lemma \ref{lem:cbd} gives the assertion.
\end{proof}

Finally, we turn our attention to the last term in \eqref{eq:defF}.
\begin{lemma}\label{lem:ddtiou2}
i) There is $C>0$ such that for any $ζ∈ℝ$ and any $ε>0$ we have
\begin{equation}\label{eq:lemddtu-new}
 \ddt \io |\uep|^2+\io |\na \uep|^2 \leq C\io \left(\nep-ζ\right)^2 
 +   C\left(\io f^{\frac65}\right)^{\frac53} \ontme.
\end{equation}
ii) Moreover, for any $\eta>0$ there is $C_{η}>0$ such that for any \epscond, 
\begin{equation}\label{eq:lemddtu}
 \ddt \io |\uep|^2+\io |\na \uep|^2 \leq \eta \io \nep^2\ln\nep 
 +   C\left(\io f^{\frac65}\right)^{\frac53}+C \ontme.
\end{equation}
\end{lemma}
\begin{proof}
If for any $ε>0$ we multiply \eqref{eq:uep} by $\uep$, we obtain 
\begin{align}\nn
 \frac12\ddt \io |\uep|^2 =& \io \uep•\uept=\io \uep•\Delta \uep + \io\uep•\na P - \io (Y_\eps\uep•\na)\uep\cdot \uep + \io \nep\na \Phi• \uep+\io \uep\cdot f\\
 =& -\io |\na \uep|^2 + \io \nep\na\Phi• \uep +\io \uep\cdot f,\label{eq:ddtueq1} \ontme,
\end{align}
where we have used that $\na\cdot\uep=0$ so that for any $ε>0$
\begin{equation*}
 \io (\Yep\uep•\na)\uep\cdot \uep= -\io \na\cdot (\Yep\uep) |\uep|^2-\frac12\io \Yep\uep\cdot \na|\uep|^2=\frac12\io (\na\cdot \Yep\uep)|\uep|^2=0 \ontme.
\end{equation*}
That $\uep$ is divergence-free also shows $\io \nep(t)\na\Phi•\uep(t)=\io \left(\nep(t)-ζ\right)\na\Phi•\uep(t)$ for any $t\in(0,\Tmaxe)$. Young's inequality in combination with Poincar\'e's inequality and the boundedness of $\na\Phi$ enables us to find $k_1>0$ such that 
\begin{equation}\label{eq:ddtueq4}
 \io |(\nep(t)-ζ)\na\Phi•\uep(t)| \leq k_1\io \left(\nep(t)-ζ\right)^2 + \frac14 \io |\na \uep(t)|^2 
\end{equation}
holds for any $t\in(0,\Tmaxe)$ for any \epscond. \\
From the embedding $W_0^{1,2}(\Om)\embeddedinto L^6(\Om)$ we can obtain a constant $k_2>0$ such that 
\begin{align*}
 \norm[L^6(\Om)]{w}\leq& k_2\norm[L^2(\Om)]{\na w} \quad \mbox{for all } w\in W^{1,2}_0(\Om)
\end{align*}
and hence H\"older's and Young's inequalities allow us to estimate 
\begin{align}\label{eq:ddtueq2}
 \io \uep\cdot f &\leq \left(\io|\uep|^6\right)^{\frac16}\left(\io |f|^{\frac65}\right)^{\frac56} \leq k_2\norm[L^2(\Om)]{\na \uep}\left(\io f^{\frac65}\right)^{\frac56}\nn\\
 &\leq \frac14\io |\na \uep|^2 + k_3\left(\io |f|^{\frac65}\right)^{\frac53}\ontme \mbox{ for all } ε>0
\end{align}
with some $k_3>0$. Adding \eqref{eq:ddtueq1}, \eqref{eq:ddtueq2} and \eqref{eq:ddtueq4} results in \eqref{eq:lemddtu-new}. 
If we furthermore use that $\io \phii^2\leq a\io \phii^2\ln\phii+|\Om|e^{\frac1a}$ for any positive function $\phii$ and any $a>0$, for any $\eta>0$ we can find $C_\eta>0$ such that 
\begin{equation}\label{eq:ddtneq3}
 \io |\nep\uep•\na\Phi|\leq \eta\io \nep^2\ln\nep +\frac14\io |\na\uep|^2 + C_\eta \ontme
\end{equation}
holds for any \epscond, thus establishing \eqref{eq:lemddtu}. 
\end{proof}

If we now amalgamate Lemma \ref{lem:ddtnlnn}, Lemma \ref{lem:ddtnacqc} and Lemma \ref{lem:ddtiou2}, we end up with 

\begin{lemma}\label{lem:ddtF}
 There are $C,k_0,K>0$ such that 
\begin{align*}
 \ddt& \left[\io\nep\ln\nep +\frac{\chi}2 \io \frac{|\na \cep|^2}{\cep} + K\chi\io |\uep|^2\right] \\
   &+\frac{\mu}4\io \nep^2\ln\nep + \io \frac{|\na\nep|^2}{\nep} + k_0\io \cep|D^2\ln\cep|^2+k_0\io \frac{|\na\cep|^4}{\cep^3} + k_0\io |\na\uep|^2 
\leq& C\left(\io f^{\frac65}\right)^{\frac53} + C
\end{align*}
on $(0,\Tmaxe)$ for all \epscond.
\end{lemma}
\begin{proof}
We fix $K$ and $k$ as in Lemma \ref{lem:ddtnacqc}, apply Lemma \ref{lem:ddtiou2} with $\eta=\frac{\mu}{4 K\chi}$ and add the inequality given by Lemma \ref{lem:ddtnlnn} to the $\frac{\chi}2$-multiple of that from Lemma \ref{lem:ddtnacqc} and $K\chi$ times the inequality from Lemma \ref{lem:ddtiou2} ii). With $k_0:=\frac{χ}2\min\set{K,k}$, Lemma \ref{lem:ddtF} results immediately. 
\end{proof}

We collect the bounds this quasi-energy inequality gives rise to:

\begin{lemma}\label{lem:allthebounds}
There is $C>0$ such that for any $ε>0$ the estimates 
\begin{align}
 &\io \nep(t)\ln\nep(t) + \io \frac{|\na\cep(t)|^2}{\cep(t)} + \io |\uep(t)|^2\leq C \quad \mbox{hold for all } t\in(0,T)\label{eq:boundsLine1}\\
 \intertext{and such that we may estimate }
 &\int_t^{t+τ}\io \nep^2\ln\nep + \int_t^{t+τ}\io \frac{|\na\nep|^2}{\nep} + \int_t^{t+τ}\io \cep|D^2\ln\cep|^2 \leq C\label{eq:boundsLine2}\\
 &\int_t^{t+τ}\io \frac{|\na\cep|^4}{\cep^3} + \int_t^{t+τ}\io |\na \uep|^2
 \leq C\label{eq:boundsLine3}\\
 &\int_t^{t+τ}\io |\na \nep|^{\frac43} + \io |\na\cep|^2+\int_t^{t+τ}\io |\na\cep|^4 +\int_t^{t+τ}\io \nep^2 \leq C\label{eq:boundsLine4}
\end{align}
for any \epscond\ and any $t\in [0,\Tmaxe-τ)$, where $τ=\min\set{1,\frac12\Tmaxe}$.
\end{lemma}
\begin{proof}
We let $\calF_{ε}(t):=\io\nep\ln\nep +\frac{\chi}2 \io \frac{|\na \cep|^2}{\cep} + K\chi\io |\uep|^2$, note that each of the summands is bounded from below, and that $s\ln s\leq \frac1{2e}+s^2\ln s$ for any $s>0$, $\io\frac{|\na\cep(t)|^2}{\cep(t)}\leq \norm[\Lom{\infty}]{\cep(t)}\io \frac{|\na\cep(t)|^4}{\cep^3(t)} + |\Om| $ for any $t\in(0,\Tmaxe)$ and that there is $C_p>0$ such that $\io |\uep|^2\leq C_p\io|\na\uep(t)|^2$ for any $t\in(0,\Tmaxe)$. Hence (and by \eqref{reg:f}), $\calF_{\eps}$ satisfies an ODI of the form $\calF'+k_1\calF\leq k_2$ and we may conlude the validity of \eqref{eq:boundsLine1}. The 
 estimates in \eqref{eq:boundsLine2} and \eqref{eq:boundsLine3} then directly result from Lemma \ref{lem:ddtF} upon integration. For \eqref{eq:boundsLine4}, we observe that the bound on $\int_t^{t+τ}\io \nep^2$ results from Lemma \ref{lem:nleq}, and that by Young's inequality 
\[
 \int_t^{t+τ}\io |\na\nep|^{\frac43}=\int_t^{t+τ}\io \frac{|\na\nep|^{\frac43}}{\nep^{\frac23}}\nep^{\frac23} \leq \frac23\int_t^{t+τ}\io \frac{|\na\nep|^2}\nep+\frac13\int_t^{t+τ}\io\nep^2.
\]
Furthermore, by Lemma \ref{lem:cbd} for any $ε>0$
\[
 \io|\na\cep(t)|^2 \leq \norm[L^\infty(\Om)]{\cep(t)} \io \frac{|\na\cep(t)|^2}{\cep(t)} \leq \norm[L^\infty(\Om)]{c_0} \io \frac{|\na\cep(t)|^2}{\cep(t)} \quad \mbox{for any } t\in(0,\Tmaxe),
\]
which is bounded due to \eqref{eq:boundsLine1}. The integral $\int_t^{t+τ}\io |\na\cep|^4$ can be treated similarly, invoking \eqref{eq:boundsLine3}. 
\end{proof}

A first consequence of these bounds is that the approximate solutions are global and we may a posteriori ignore any condition of the type $t<\Tmaxe$ in the previous lemmata.
\begin{lemma}\label{lem:epglobal}
 For any $ε>0$, $\Tmaxe=\infty$. 
\end{lemma}
\begin{proof}
 Under the assumption that $\Tmaxe<\infty$, for any $ε>0$, Lemma \ref{lem:allthebounds} would provide us with $C>0$ such that 
\[
 \int_0^{\Tmaxe}\!\!\io |\na\cep|^4\leq C \quad \mbox{and} \quad \io |\uep|^2 \leq C \quad \mbox{on } (0,\Tmaxe). 
\]
With this as starting point, we could follow the reasoning of \cite[Lemma 3.9]{wk_ctfluid3dnastoexist} to derive a contradiction to \eqref{eq:continuationcrit}.
There differential inequalities for  $\io \nep^4$ and $\io |A^{\frac12}\uep|^2$ first yielded bounds for these quantities on $[0,\Tmaxe)$, then smoothing estimates for the Stokes semigroup (if combined with an embedding for the domains of fractional powers of $A$) and for the Neumann heat semigroup led to estimates for $\norm[L^{∞}(\Om\times(0,\Tmaxe))]{\uep}$, $\norm[L^{∞}((0,\Tmaxe),\Lom4)]{\na\cep}$ and $\norm[L^{∞}(\Om\times(0,\Tmaxe))]{\nep}$.
\end{proof}

\subsection{Time regularity}

In preparation of an Aubin-Lions type compactness argument, we intend to supplement Lemma \ref{lem:allthebounds} with bounds on time-derivatives. This will be the purpose of the following three lemmata. 

\begin{lemma}\label{lem:nept}
For any $T>0$ there is $C>0$ such that for every \epscond, 
\[
 \norm[L^1((0,T),(W_0^{2,4}(\Om))^*)]{\nept}\leq C.
\]
\end{lemma}
\begin{proof} Let $T>0$.
We recall that $\left(L^1((0,T),(W_0^{2,4}(\Om))^*)\right)^*=L^\infty((0,T),W_0^{2,4}(\Om))$ and that hence
\[
 \norm[L^1((0,T),(W_0^{2,4}(\Om))^*)]{w} = \sup\setl{\intnT\io w\phii;\; \phii\in L^\infty((0,T),W_0^{2,4}(\Om)), \norm[L^\infty((0,T),W_0^{2,4}(\Om))]{\phii}\leq 1},
\]
and introduce $k_1>0$ such that 
\begin{align*}
 \max&\setl{\norm[\LT4{W^{1,4}(\Om)}]{φ},\norm[\LT2{\Lom2}]{φ},\norm[\LT2{W^{1,2}(\Om)}]{φ},
 \norm[\LT{∞}{W^{1,∞}(\Om)}]{φ} }\\
&\qquad \qquad \leq k_1\norm[L^\infty((0,T),W_0^{2,4}(\Om))]{\phii} \qquad \mbox{for all } \phii\in L^\infty((0,T),W_0^{2,4}(\Om)), 
\end{align*}
which is guaranteed to exist by the continuous embeddings of $W^{2,4}(\Om)$ into $W^{1,4}(\Om)$, $\Lom2$, and $W^{1,∞}(\Om)$.
We then pick an arbitrary $\phii\in L^\infty((0,T),W_0^{2,4}(\Om))$ having norm $\norm[L^\infty((0,T),W_0^{2,4}(\Om))]{φ}\leq 1$ and test \eqref{eq:nep} by $\phii$, so that we obtain 
\begin{align*}
 \intnT\io \nept\phii=&\intnT\io \na\nep \cdot \na\phii - \chi\intnT\io \frac{\nep}{1+\eps\nep} \na\cep\cdot\na\phii \\
&+ \kappa\intnT\io\nep\phii - \mu\intnT\io \nep^2\phii -\intnT\io \nep\uep•\na\phii\\
\leq&\norm[\LT4{\Lom4}]{\na\phii}\left(\intnT\io|\na\nep|^{\frac43}\right)^{\frac34} \\
&+ \chi\norm[\LT2{\Lom2}]{\na\phii}\left(\intnT\io\nep^2\right)^{\frac12} \sup_{t\in(0,T)}\left(\io|\na\cep (t)|^2\right)^{\frac12}\\&
+\kappa\norm[\LT2{\Lom2}]{\phii}\left(\intnT\io\nep^2\right)^{\frac12}+\mu\norm[\LT{\infty}{\Lom\infty}]{\phii}\intnT\io \nep^2 \\
&+ \norm[\LT{\infty}{\Lom\infty}]{\na\phii}\left(\intnT\io |\uep|^2\intnT\io\nep^2\right)^{\frac12} \qquad \mbox{ for all } ε>0.
\end{align*}
If we let $C$ be as in Lemma \ref{lem:allthebounds}, we obtain
\[
  \intnT\io \nept\phii \leq k_1 (CT)^{\frac34}+χk_1(CT)^{\frac12}C^{\frac12}+κk_1C^{\frac12}+μk_1CT+k_1CT \qquad \mbox{ for all } ε>0
\]
and thus conclude the proof.
\end{proof}

We continue with a similar statement concerning the second component of the solution. 
\begin{lemma}\label{lem:cept}
For all $T>0$ there is $C>0$ such that 
\[
 \norm[L^2((0,T),(W^{1,2}_0(\Om))^*)]{\cept}\leq C
\]
for all $ε>0$.
\end{lemma}
\begin{proof}
Let $T>0$. Employing Hölder's inequality and using that for any $ε>0$ and $s>0$ apparently $\frac{\ln(1+\eps s)}{\eps}\leq s$, we see that for any $\phii\in L^2((0,T),W^{1,2}_0(\Om))$ 
satisfying $\norm[\LT2{W^{1,2}(\Om)}]{\phii}\leq1$ we have 
\begin{align*}
 &\intnT\io\cept\phii\\&=\intnT\io\na\cep\cdot\na\phii-\intnT\io\cep\frac1\eps\ln(1+\eps\nep)\phii - \intnT\io \cep\uep•\na\phii\\
 \leq&\left(\intnT\io|\na\cep|^2\right)^{\frac12}\norm[\LT2{\Lom2}]{\na\phii} + \norm[L^\infty(\Om\times(0,T))]{\cep}\left(\intnT\io\nep^2\right)^{\frac12}\norm[\LT2{\Lom2}]{\phii} \\
 &+ \norm[L^\infty(\Om\times(0,T))]{\cep}\left(\intnT\io |\uep|^2\right)^{\frac12}\norm[\LT2{\Lom2}]{\na\phii},
\end{align*}
again concluding the proof with the aid of Lemma \ref{lem:allthebounds}.
\end{proof}

\begin{lemma}\label{lem:uept}
For all $T>0$ there is $C>0$ such that 
\begin{equation}\label{eq:lemuept}
 \norm[L^2((0,T),(W^{1,3}(\Om))^*)]{\uept}\leq C
\end{equation}
holds for any \epscond.
\end{lemma}
\begin{proof}
Let $\psi\in L^2(0,T,W^{1,3}(\Om))$ with $\norm[L^2(0,T,W^{1,3}(\Om))]{\psi}=1$. 
Then 
\begin{align}\label{eq:tobound-uept}
&\nn  \intnT\!\!\!\io \uept •\psi \leq \left| -\intnT\!\!\!\io \na\uep•\na\psi + \intnT\!\!\!\io \Yep\uep\otimes\uep \na\psi + \intnT\!\!\!\io \nep\na\Phi\cdot\psi+\intnT\!\!\!\io f\cdot\psi\right|\\
\leq& \norm[\LT2{\Lom2}]{\na \uep}\norm[\LT2{\Lom2}]{\na\psi} + \norm[\LT2{\Lom6}]{\Yep\uep} \norm[\LT{∞}{\Lom2}]{\uep}\norm[\LT2{\Lom3}]{\na\psi} \nn\\
&+ \norm[L^\infty(\Om)]{\na\Phi} \norm[\LT2{\Lom2}]{\nep}\norm[\LT2{\Lom2}]{\psi}+\norm[\LT2{\Lom{\frac65}}]{f}\norm[\LT2{\Lom6}]{\psi}.
\end{align}
Here we can use that by the embedding $W^{1,2}_{σ}(\Om)\embeddedinto L^6(\Om)$ and nonexpansiveness of $\Yep$ on $L^2_\sigma$ (see e.g. \cite[(II.3.4.6)]{sohr_book}) there is $k>0$ such that for any \epscond 
\[
 \norm[\Lom6]{\Yep\uep}\leq k\norm[\Lom2]{\na\Yep\uep}=k\norm[\Lom2]{A^{\frac12}\Yep\uep}=k\norm[\Lom2]{\Yep A^{\frac12}\uep}\leq k\norm[\Lom2]{A^{\frac12}\uep}=k\norm[\Lom2]{\na\uep}. 
\]
Thus, the bounds on $\norm[\LT2{\Lom2}]{\na\uep}$, $\norm[\LT{\infty}{\Lom2}]{\uep}$, $\norm[\LT2{\Lom2}]{\nep}$ from Lemma \ref{lem:allthebounds} entail boundedness of the expression in \eqref{eq:tobound-uept}, so that \eqref{eq:lemuept} results.
\end{proof}

\subsection{Passing to the limit. Proof of Theorem \ref{thm:exweaksol}}

With these lemmata we have collected sufficiently many estimates to construct weak solutions by compactness arguments. Before doing so, let us define what a weak solution is supposed to be:

\begin{definition}\label{def:weaksol}
 A weak solution of \eqref{eq:system} is a triple $(n,c,u)$ of functions such that 
\begin{align*}
 n&\in L^2_{loc}([0,\infty),L^2(\Om))\cap L^{\frac43}_{loc}([0,\infty),W^{1,\frac43}(\Om))\\
 c&\in L^2_{loc}([0,\infty),W^{1,2}(\Om))\\
 u&\in L^2_{loc}([0,\infty),W^{1,2}_{0,σ}(\Om)) 
\end{align*}
and that 
\begin{align*}
 -\intnT\!\!\!\!\io\!\!\! n\phii_t-\io\!\!\! n_0\phii(\cdot,0)-\intnT\!\!\!\!\io\!\!\! nu•\na\phii &= -\intnT\!\!\!\!\io\!\!\! \na n•\na \phii + \chi\intnT\!\!\!\!\io\!\!\! n\na c•\na \phii+\kappa\intnT\!\!\!\!\io\!\!\! n\phii - \mu\intnT\!\!\!\!\io\!\!\! n^2\phii\\
 -\intnT\!\!\!\!\io\!\!\! c\phii_t -\io\!\!\! c_0\phii(\cdot,0)-\intnT\!\!\!\!\io\!\!\! cu•\na\phii &= -\intnT\io\na c•\na \phii -\intnT\io nc\phii\\
 -\intnT\!\!\!\!\io\!\!\! u•\psi_t - \io\!\!\! u_0•\psi(\cdot,0) - \intnT\!\!\!\!\io\!\!\! u\otimes u•\na\psi &= -\intnT\!\!\!\!\io\!\!\!\na u•\na\psi + \intnT\!\!\!\!\io\!\!\! n\na\psi\na\Phi + \intnT\!\!\!\!\io\!\!\! f\cdot\psi
\end{align*}
hold for any $\phii\in C_0^\infty([0,\infty)\times\Ombar)$ and any $\psi\in C_{0,\sigma}^\infty([0,\infty)\times\Ombar)$, respectively.
\end{definition}

Such weak solutions do exist:

\begin{proposition}\label{prop:ex-weaksol}
There exist a sequence $(\eps_j)_{j\nat} \downto 0$ and functions $n,c,u$ such that $n\in L^2_{loc}([0,\infty),L^2(\Om))\cap L^{\frac43}_{loc}([0,\infty),W^{1,\frac43}(\Om))$, $c\in L^2_{loc}([0,\infty),W^{1,2}(\Om))$, $u\in L^2_{loc}([0,\infty),W^{1,2}_{0,σ}(\Om))$ and that 
\begin{align}
 \nep &\to n &\qquad &\mbox{in } L^{\frac43}_{loc}([0,\infty),L^p(\Om)) \quad &&\mbox{for all } p\in [1,\frac{12}5) \quad && \mbox{ and a.e. in }\Om\times(0,\infty), \label{conv:n}\\
 \cep &\to c &\qquad &\mbox{in } C^0_{loc}([0,\infty),L^p(\Om)) \quad &&\mbox{for all } p\in[1,6)\quad&& \mbox{ and a.e. in }\Om\times(0,\infty),\label{conv:c}\\
 \cep &\weakstarto c &\qquad&\mbox{in } L^{\infty}(\Om\times(t,t+1)) \quad &&\mbox{for all }t\ge0, \label{conv:cinfty}\\
 \uep&\to u &\qquad &\mbox{in } L^2_{loc}([0,\infty),L^p(\Om))\quad &&\mbox{for all } p\in[1,6)\quad&& \mbox{ and a.e. in }\Om\times(0,\infty),\label{conv:u}\\
 \na\nep&\weakto \na n&\qquad &\mbox{in } L^{\frac43}_{loc}([0,\infty),L^{\frac43}(\Om)),\label{conv:nan}\\
 \na\cep&\weakstarto \na c&\qquad & \mbox{in } L^\infty_{loc}([0,\infty),L^2(\Om)),\label{conv:nac}\\
 \na\uep&\weakto\na u&\qquad & \mbox{in } L^2_{loc}([0,\infty),L^2(\Om)),\label{conv:nau}\\
 \Yep \uep &\to \uep&\qquad&\mbox{in } L^2_{loc}([0,\infty),L^2(\Om)),\label{conv:Yu}\\
 \nep&\to n&\qquad & \mbox{in } L^2_{loc}([0,\infty),L^2(\Om)), \label{conv:nl2}
\end{align}
as $\eps=\eps_j\downto 0$ and such that $(n,c,u)$ form a weak solution to \eqref{eq:system} in the sense of Definition \ref{def:weaksol}.
\end{proposition}
\begin{proof}
For any $p\in[1,\frac{12}5)$, $W^{1,\frac43}(\Om)\embeddedinto\embeddedinto \Lom p\embeddedinto (W_0^{2,4}(\Om))^*$, so that for any $T>0$ the bound on $\norm[L^{\frac43}([0,T),W^{1,\frac43}(\Om))]{\nep}$ from Lemma \ref{lem:allthebounds}, which was independent of $ε$, together with Lemma \ref{lem:nept} and \cite[Corollary 4]{simon} shows relative compactness of $\set{\nep; \eps>0}$ 
in $L^{\frac43}([0,T),L^p(\Om))$ and thus ensures the existence of a sequence $(\eps_j)_j$ satisfying \eqref{conv:n}. Because for any $T>0$ there is a uniform bound on \(\intnT\io Ψ(\nep^2)\) for $Ψ(x)=\frac x2 \ln(x)\), $\set{\nep^2; \eps>0}$ is weakly relatively precompact in $L^1(\Om\times(0,T))\) by the Dunford-Pettis theorem (cf. \cite[Thm. IV.8.9]{dunford_schwartz_I_58}) 
and hence, along a subsequence of $(ε_j)_{j\nat}$, $\nep^2\weakto z$ in $L^1(\Om\times(0,T))$ for some $z\in L^1(\Om\times(0,T))$, where $z$ has to coincide with $n^2$ due to \eqref{conv:n}. In particular, $\intnT\io \nep^2 \to \intnT\io n^2$ as $ε_j\to 0$. Since moreover, along a further subsequence, $\nep\weakto n$ in $L^2_{loc}([0,∞),L^2(\Om))$ due to \eqref{eq:boundsLine4}, we obtain \eqref{conv:nl2}.
Similarly, bounds on $\cep$ with respect to the norm of $L^\infty([0,T),W^{1,2}(\Om))$ and on $\cept$ in $L^2([0,T),(W^{1,2}(\Om))^*)$ as obtained in Lemma \ref{lem:allthebounds} and Lemma \ref{lem:cept} and the embedding $W^{1,2}(\Om)\embeddedinto\embeddedinto L^p(\Om)\embeddedinto (W_0^{1,2}(\Om))^*$ for all $p\in[1,6)$ allow for an application of \cite[Corollary 4]{simon}, which yields \eqref{conv:c} along a suitable subsequence of the sequence previously found.
Similar reasoning for $u$, combining bounds on $\uep$ in $L^2([0,T),W^{1,2}_\sigma(\Om))$ and on $\uept$ in $L^2([0,T),(W^{1,3}_\sigma(\Om))^*)$, results in \eqref{conv:u}. 
Due to \eqref{conv:u}, also for almost every $t>0$ we have $\uep(\cdot,t)\to u(\cdot,t)$ and taking into account \cite[II.(3.4.6)]{sohr_book} 
and \cite[II.(3.4.8)]{sohr_book} 
shows that $\norm[L^2(\Om)]{\Yep\uep-u}\leq\norm[L^2(\Om)]{\Yep\uep-\Yep u+\Yep u- Yu}\leq \norm[L^2(\Om)]{\uep-u}+\norm[L^2(\Om)]{(\Yep-Y)u}\to 0$ for a.e. $t>0$. Since $\norm[L^2(\Om)]{\Yep\uep}\leq\norm[L^2(\Om)]{\uep}$ and $\norm[L^2(\Om)]{\uep}$ converges in $L^2((0,T))$, a version of Lebesgue's theorem ensures the validity of \eqref{conv:Yu}. 
Convergence of the gradients along further subsequences, as asserted in \eqref{conv:nan}, \eqref{conv:nac}, \eqref{conv:nau}, is easily obtained from the bounds given in Lemma \ref{lem:allthebounds}. 
The convergence properties asserted in \eqref{conv:c}, \eqref{conv:u}, \eqref{conv:nan}, \eqref{conv:nac},\eqref{conv:nau}, \eqref{conv:Yu}, \eqref{conv:nl2} finally, are sufficient to pass to the limit in each integral making up a weak formulation of system \eqref{epsys}, so that $(n,c,u)$ is a weak solution to \eqref{eq:system}. 
\end{proof}

The most important consequence of this proposition is the following: 

\begin{proof}[Proof of Theorem \ref{thm:exweaksol}]
 The theorem is part of the statement proven by Proposition \ref{prop:ex-weaksol}.
\end{proof}

\section{Eventual smoothness and asymptotics}\label{sec:thm2}

\subsection{Lower bound for the bacterial mass}
Although we already know an upper bound for $\io n$, we are still lacking a corresponding estimate from below, which was crucial in the derivation of the convergence of $c$ in \cite{wk_ns_oxytaxis}. 
Consideration of the function
\[
 \calG_{ε,B}(t):= 
 \io\nep(t)
 - \frac\kappa\my \io\ln \frac{\my \nep(t)}{\kappa} + \frac B2\io \cep^2(t), \qquad t\in(0,\infty)
\]
will help us to recover this lower bound. At the same time, we will obtain another cornerstone for the proof of convergence of $n$ (see Lemma \ref{lem:intintnp} and Lemma \ref{lem:convergence}).

In \cite{wk_ksns_logsource}, 
a similar functional has been employed to obtain convergence of $n$ and $c$ to a constant equilibrium. The model considered there contains the Keller-Segel equation as second equation and due to the contributions of the production term $+n$ therein, whose influence is increased with increasing values of $B$, it was not possible to choose $B$ arbitrarily large there, which in the end resulted in a largeness condition on $\my$ (\cite[(8.3)]{wk_ksns_logsource}).
Thanks to the consumption term in \eqref{eq:nep}, all terms obtained from this equation work in favour of our estimate and we do not need a corresponding condition on $\my$ and can choose $B$ in such a way that $\calG_{ε,B}$ becomes an energy functional.

\begin{lemma}\label{lem:ddtlog}
 There is $B_0$ such that for any $B>B_0$ and any \epscond, we have
\begin{equation}\label{eq:ddtG}
 \ddt \calG_{ε,B}(t) + \my \io\left(\nep-\frac\kappa\my\right)^2+\frac{\kappa}{2\my} \io \frac{|\na\nep|^2}{\nep^2} \leq 0 \onni. 
\end{equation}
\end{lemma}
\begin{proof}
Let $B_0:= \frac{\kappa\chi^2}{2\my}$ and $B>B_0$. The derivative of $\calG_{ε,B}$ then satisfies 
 \begin{align*}
  \ddt \calG_{ε,B} =& \io \nept-\frac{\kappa}\my \io \frac{\nept}{\nep} + B\io \cep\cept\\
=& \kappa\io \nep -\my \io \nep^2 - \frac\kappa\my \io \frac{\Lap \nep}{\nep} + \frac\kappa\my \chi \io \frac{\na\cdot\kl{\frac{\nep}{1+\eps\nep}\na\cep}}{\nep} - \frac {\kappa^2}\my \io 1 + \kappa\io \nep +\frac\kappa\my \io \uep•\frac{\na\nep}{\nep}\\
 &+ B\io \cep\Lap \cep - B\io\cep^2\frac1\eps\ln(1+\eps\nep)-B\io \cep\na\cep\cdot u \fate.
 \end{align*}
Here we can use that $\uep$ is divergence-free and hence integration by parts shows that $\io \na(\frac12 \cep^2)\cdot \uep$ vanishes as well as $\io \uep\cdot \na \ln \nep$. Furthermore we can summarize the terms without derivatives according to 
\[
 -\my\io\nep^2-\frac{\kappa^2}\my\io1+2\kappa\io\nep=-\my\io\kl{\nep-\frac\kappa\my}^2 \fate
\]
so that for all $ε>0$ we obtain 
\begin{align*}
 \ddt \calG_{ε,B} =& -\my\io\kl{\nep-\frac\kappa\my}^2 - \frac\kappa\my\io\frac{|\na\nep|^2}{\nep^2}+\frac{\kappa\chi}\my\io\frac{\nep\na \cep}{(1+\eps\nep)\nep^2}\cdot\na \nep\\ &\quad - B\io |\na \cep|^2- B\io\cep^2\frac1\eps\ln(1+\eps\nep)
\end{align*}
on $(0,\infty)$. Nonnegativity of $\io \cep^2\frac1\eps\ln(1+\eps\nep)$ and an application of Young's inequality together with the trivial estimate $\frac{\nep}{(1+\eps\nep)\nep}\leq 1$ yield
\begin{align*}
 \ddt \calG_{ε,B} \leq&-\my\io\kl{\nep-\frac\kappa\my}^2-\frac\kappa\my\io\frac{|\na \nep|^2}{\nep^2}+\frac\kappa{2\my}\io\frac{|\na\nep|^2}{\nep^2}+\frac{\kappa\chi^2}{2\my}\io|\na\cep|^2-B\io|\na\cep|^2
\end{align*}
on $(0,\infty)$ for any $ε>0$, so that we finally arrive at \eqref{eq:ddtG}.
\end{proof}

We collect the estimates implicitly contained in Lemma \ref{lem:ddtlog}:
\begin{lemma}\label{lem:estimatesfromddtG}
 There are $k>0$ and $C>0$ such that for any \epscond
\begin{align}
 \io \nep(\cdot,t)>k\qquad \mbox{for all } t>0 \label{eq:ionepgeq}\\
 \intninf\!\! \io \left(\nep-\frac{\kappa}\mu\right)^2\leq C\label{eq:iinminusgwqleq}.
\end{align}
\end{lemma}
\begin{proof}
 Let $B>B_0$ with $B_0$ as in Lemma \ref{lem:ddtlog} and $ε>0$. For any $t>0$. integration of \eqref{eq:ddtG} on $(0,t)$ yields 
\begin{align*}
 \io \nep(t)&-\frac\kappa\mu\io \ln\frac{\mu \nep(t)}\kappa + \frac{B}2\io\cep^2(t)+\mu\intnt\io\left(\nep-\frac\kappa\mu\right)+\frac{\kappa}{2\mu}\intnt\io \frac{|\na\nep|^2}{\nep^2} \\
&\leq \io n_0-\frac\kappa\mu\io\ln\frac{\mu n_0}\kappa + \frac{B}2\io c_0^2.
\end{align*}
In particular, for all $t>0$, 
\begin{align}\nn
 \frac\kappa\mu\io\ln\frac{\mu\nep(t)}{\kappa} &\geq \io \nep(t)-\io n_0+\frac B2\io \cep^2(t)-\frac B2\io c_0^2+\frac{\kappa}\mu\io\ln\frac{\mu n_0}\kappa \\ 
&+ \mu\intnt\io \left(\nep-\frac\kappa\mu\right)^2+\frac{\kappa}{2\mu}\intnt\io\frac{|\na\nep|^2}{\nep^2}\nn\\
 &\geq -\io n_0 - \frac B2 \io c_0^2+\frac{\kappa}\mu\io\ln\frac{\mu n_0}\kappa + \mu\intnt\io \left(\nep-\frac\kappa\mu\right)^2+\frac{\kappa}{2\mu}\intnt\io\frac{|\na\nep|^2}{\nep^2}\label{eq:firstestioln}
\end{align}
Since $\frac\kappa\mu\io\ln\frac{\mu\nep(t)}{\kappa}\leq \frac\kappa\mu \io \frac{\mu \nep(t)}\kappa=\io \nep(t)$ is bounded according to Lemma \ref{lem:nleq}, this entails \eqref{eq:iinminusgwqleq}. 
The estimate in \eqref{eq:firstestioln} also shows that 
\[
 \io \ln\frac{\mu \nep(t)}{\kappa} \frac{1}{|\Om|} \geq \frac{\mu}{|\Om|\kappa}\left[-\io n_0 - \frac B2 \io c_0^2 +\frac\kappa\mu\io \ln \frac{\mu n_0}{\kappa} \right]=:k_1 \qquad \mbox{for all }t>0,
\]
so that Jensen's inequality implies 
\[
 \frac{\mu}{\kappa}\io \nep(t) \frac{1}{|\Om|}\geq e^{\io \ln\frac{\mu \nep(\cdot,t)}{\kappa} \frac1{|\Om|}}\geq e^{k_1}\fat
\]
and hence \eqref{eq:ionepgeq}.
\end{proof}

\subsection{Decay of oxygen}
With the lower bound on the bacterial mass from Lemma \ref{lem:estimatesfromddtG} we are well-equipped for the derivation of decay of $c$ by means of \eqref{eq:cep}. Smallness of $c$ will play an important role in Section \ref{sec:bdnessofn}, when we derive bounds on $\nep$ in higher $L^p$-norms via a differential inequality holding for small values of $\cep$ only. For turning such bounds into information on $n$, it will be crucial that the validity of the ODI does not hinge on $\eps$ too much, i.e. that the decay of $\cep$ be uniform in $\eps$.
In pursuance of this uniformity, in the following lemma we will consider $c$ instead of $\cep$ and afterwards carry back the decay information to the $\cep$ (which, due to their differentiability, are much better suited for making an appearance in ODIs like that in the proof of Lemma \ref{lem:nepbd}).
The idea of the proof of boundedness of $\cep$ is taken from \cite[Sec. 4]{wk_ns_oxytaxis}.

\begin{lemma}\label{lem:iintcto0}
 For any $\eta>0$ there is $T>0$ such that for any $t>T$
\[
 0\leq \intttpe\!\!\!\io c <\eta.
\]
\end{lemma}
\begin{proof}
 Integrating \eqref{eq:cep} shows that 
\[
 \io c_0 \geq \intnt\io \cep\frac{\ln(1+\eps\nep)}{\eps} \qquad \mbox{for any }t>0 \quad \mbox{ and any }ε>0, 
\]
which in light of \eqref{conv:c}, \eqref{conv:n} asserts that 
\[
 \io c_0 \geq \intnt\io nc \qquad \mbox{for all }t>0
\]
and hence in particular 
\begin{equation}
 \intttpe\!\!\!\io nc\to 0 \qquad \mbox{as }t\to\infty \label{eq:conv-inttpeionc}.
\end{equation}
Denoting the average value $\frac1{|\Om|}\io c(\cdot,t)$ of $c(\cdot,t)$ by $\cbar(\cdot,t)$ we observe that 
\begin{equation}\label{eq:splitiintnc}
 \intttpe\!\!\!\io nc=\intttpe\!\!\!\io n(c-\cbar)+\intttpe \cbar \io n \fat,
\end{equation}
where 
\[
 \intttpe\!\!\!\io n(c-\cbar) \leq \left(\intttpe\!\!\!\io n^2\right)^{\frac12} \left(\intttpe\!\!\!\io (c-\cbar)^2\right)^{\frac12} \leq k_1^{\frac12}\left(c_p\intttpe\!\!\!\io |\na c|^2\right)^{\frac12}
\]
with $c_p$, $k_1$ being constants obtained from Poincar\'e's inequality and Lemma \ref{lem:nleq} in combination with \eqref{conv:nl2}. We use $k_2$ to denote the positive lower bound for $\io n$, which is guaranteed to exist by Lemma \ref{lem:estimatesfromddtG} and \eqref{conv:n}. 
Since $\na c\in L^2(\Om\times(0,\infty))$ due to Lemma \ref{lem:intnac2}, $\intttpe\!\!\!\io n(c-\cbar)\to 0$ as $t\to\infty$ and taking \eqref{eq:conv-inttpeionc} and \eqref{eq:splitiintnc} into account, we see that 
\[
 0\leq \frac1{|\Om|} \intttpe\!\!\!\io c k_2 \leq \intttpe\cbar\io n\to 0\qquad \mbox{as } t\to\infty. \qedhere
\]
\end{proof}

We transfer this information back to the functions $\cep$:

\begin{corollary}\label{cor:iintcepto0}
 For any $\eta>0$ there is $T>0$ such that for all $t\ge T$ there is $ε_0>0$ such that for any $\eps\in(0,\eps_0)$
\[
 \intttpe \io \cep <\eta.
\]
\end{corollary}
\begin{proof} 
 This directly results from Lemma \ref{lem:iintcto0} and \eqref{conv:c}.
\end{proof}

\begin{lemma}\label{lem:clinto0}
 For any $\eta>0$ there are $T>0$ and $\eps_0>0$ such that for every $t>T$ and every $\eps\in(0,\eps_0)$ we have 
\[
 \norm[L^\infty(\Om)]{\cep(\cdot,t)}<\eta.
\]
\end{lemma}
\begin{proof}
 The \GNI\ asserts the existence of $k_1>0$ such that 
\begin{equation}\label{eq:GNI-clinto0} 
 \norm[\Liom]{\phi}\leq k_1\left(\norm[L^4(\Om)]{\na \phi}^{\frac{12}{13}}\norm[L^1(\Om)]{\phi}^{\frac1{13}}+\norm[L^1(\Om)]{\phi}\right) \qquad \mbox{for all } \phi\in W^{1,4}(\Om)
\end{equation}
and according to Lemma \ref{lem:allthebounds} there is $k_2>0$ such that 
\begin{equation}\label{eq:linto0-nacep4}
 \intttpe\!\!\!\io |\na\cep|^4\leq k_2 \qquad \mbox{for all } t>0, \epscond.
\end{equation}
Let $\eta>0$. Let $\delta>0$ be such that $k_1k_2^{\frac3{52}}\delta^{\frac1{13}}+k_1\delta<\eta$. 
Due to Corollary \ref{cor:iintcepto0} there are $T_0>0$ and $\eps_0>0$ such that for every $\eps\in(0,\eps_0)$ we have $\int_{T_0}^{T_0+1}\!\!\io \cep<\delta$. 
Invoking \eqref{eq:GNI-clinto0} and \eqref{eq:linto0-nacep4}, we see that 
\begin{align*}
 \int_{T_0}^{T_0+1}\norm[\Liom]{\cep}\leq& k_1\int_{T_0}^{T_0+1}\norm[\Lom4]{\na\cep}^{\frac{12}{13}}\norm[\Lom1]{\cep}^{\frac1{13}}+k_1 \int_{T_0}^{T_0+1} \norm[\Lom1]{\cep}\\
 \leq& k_1\left(\int_{T_0}^{T_0+1}\norm[\Lom4]{\na\cep}^4\right)^{\frac3{13}}\left(\int_{T_0}^{T_0+1}\norm[\Lom1]{\cep}\right)^{\frac1{13}}\cdot 1^{\frac9{13}} + k_1\delta\\
 \leq& k_1 k_2^{\frac3{52}}\delta^{\frac1{13}}+k_1\delta\le \eta \qquad \mbox{for any } \eps\in(0,\eps_0). 
\end{align*}
In particular, for every $\eps\in(0,\eps_0)$, there is at least one $t_0\in[T_0,T_0+1]$ such that $\norm[\Lom\infty]{\cep(\cdot,t_0)}\leq \eta$ and thus, due to monotonicity of $\cep$ (Lemma \ref{lem:cbd}), for all $\eps\in(0,\eps_0)$ and all $t>T:=T_0+1$, 
\[
 \norm[\Lom\infty]{\cep(\cdot,t)}\le \eta. \qedhere
\]
\end{proof}

\begin{corollary}\label{cor:ctozero}
 The function $c$ obtained in Proposition \ref{prop:ex-weaksol} satisfies 
\[
 \norm[\Lom{\infty}]{c(\cdot,t)}\to 0. 
\]
\end{corollary}
\begin{proof}
 Combining Lemma \ref{lem:clinto0} with \eqref{conv:cinfty} this convergence statement results immediately.
\end{proof}

\subsection{Boundedness of $n$}\label{sec:bdnessofn}
In obtaining eventual smoothness and convergence of the solutions constructed in Proposition \ref{prop:ex-weaksol}, we will heavily rely on estimates for higher norms of $n$. 
We can achieve those for large times in Lemma \ref{lem:nepbd} and prepare this by deriving a differential inequality for $y_{\eps}(t):=	\io \frac{\nep^p}{(η-\cep)^\theta}$, which holds for small values of $\cep$. Fortunately, we already have established that $\norm[L^\infty(\Om)]{\cep(t)}$ converges to $0$. 

The same quantity has proven useful in the derivation of estimates for $\norm[L^p(\Om)]{n(t)}$ for large $t$ already in \cite[Sec. 5]{wk_arma} 
and \cite[Sec. 5]{wk_ns_oxytaxis}. Note, however, that there (that is: in the setting without logistic source) the analogue of \eqref{eq:comparewithearlier} below would read 
\[
 y_{ε}' + \left(\frac{[2p\theta+\chi p(p-1)\eta]^2}{4(\theta(1+\theta)-\chi p \theta \eta)} - p(p-1) \right)\io \frac{\nep^{p-2}|\na \nep|^2}{(\eta-\cep)^\theta} \leq 0, 
\]
so that the right hand side already equals zero, and hence at the same time bounds on $\intttpe\!\!\!\io \frac{\nep^{p-2}|\na \nep|^2}{(\eta-\cep)^\theta} $ could be obtained.

The fact that $y$ is defined on intervals $(T,∞)$ for large $T$ only, raises the problem that the initial values $y_\eps(T)$ are unknown and differ for varying $ε$. 
Fortunately, the nonlinear absorptive term allows for comparison with solutions 'starting from initial data $∞$' (see \eqref{eq:uppersoln}), so that the bound on $y_\eps$ does not depend on $\io \nep^p(T)$ and hence not on $\eps$.

Also, it is important to note that $T$ may (and will) depend on $p$, but is independent of $\eps$ due to the uniformity of the decay of $\cep$ asserted by Lemma \ref{lem:clinto0}. This will be decisive when transferring the bounds on $\norm[\Lom p]{\nep}$ to $\norm[\Lom p]{n}$.

\begin{lemma}\label{lem:nepbd}
For any $p∈(1,∞)$ there are $T^\star>0$, $ε_0>0$ and $C>0$ such that 
\[
 \io \nep^p(\cdot,t)\leq C 
\]
for all $t>T^\star$ and all $\eps\in(0,ε_0)$.
\end{lemma}
\begin{proof}
Let $p> 1$.
First fix $\theta>0$ so small that 
\[
 4p^2\theta+4p^2\theta\chi(p-1)+\chi^2p^2(p-1)^2\theta < 2 p(p-1)
\]
and let $0<\eta<\min\set{1,\theta,\frac1{2p\chi}}$. Then 
\[
 4p^2\theta+4p^2\theta\chi(p-1)\frac{\eta}{\theta} +\chi^2p^2(p-1)^2\theta\frac{\eta^2}{\theta^2}<4p(p-1)[1+\theta-\chi p\eta]
\]
and hence 
\[
 4p^2\theta^2+4p^2\theta\chi(p-1)\eta+\chi^2p^2(p-1)^2 \eta^2<4p(p-1)\theta[1+\theta-\chi p\eta],
\]
that is
\begin{equation}\label{eq:reasonforchoicethetaeta}
 (2p\theta+\chi p(p-1)\eta)^2<p(p-1)4\theta(1+\theta-\chi p\eta).
\end{equation}

We use Lemma \ref{lem:clinto0} to fix $T>0$ and $ε_0>0$ such that for any $\eps\in(0,ε_0)$, $t>T$, we have 
\begin{equation}\label{eq:csmall}
 \norm[L^\infty(\Om)]{\cep(\cdot,t)}\leq \frac\eta2.
\end{equation}

Let $\eps\in(0,ε_0)$.
Then 
\[
 y_{ε}(t):=\io \frac{\nep^p}{(\eta-\cep)^\theta}, \qquad t\ge T, 
\]
is well-defined and we can compute

\begin{align}\label{eq:ddtionpetactheta}
 \ddt \io \frac{\nep^p}{(\eta-\cep)^\theta} = & p\io \frac{\nep^{p-1}\nept}{(\eta-\cep)^\theta} + \theta\io \frac{\nep^p}{(\eta-\cep)^{1+\theta}}\cept\nn\\
 =& p\io \frac{\nep^{p-1}\Delta \nep}{(\eta-\cep)^\theta} - \chi p\io \frac{\nep^{p-1}\na•(\frac{\nep}{1+\eps \nep}\na \cep)}{(\eta-\cep)^\theta} - p\io \frac{\nep^{p-1} \uep\cdot \na \nep}{(\eta-\cep)^\theta} \nn\\
& +p\kappa\io \frac{\nep^p}{(\eta-\cep)^\theta} - p\mu\io \frac{\nep^{p+1}}{(\eta-\cep)^\theta}\nn\\
 & + \theta \io \frac{\nep^p\Delta \cep}{(\eta-\cep)^{\theta+1}} - \theta \io \frac{\nep^p\frac{\cep}\eps \ln(1+\eps \nep)}{(\eta-\cep)^{1+\theta}} - \theta\io \frac{\nep^p \uep\cdot \na \cep}{(\eta-\cep)^{1+\theta}}, \qquad \mbox{on } (T,\infty). 
\end{align}
Here we use that, since $\uep$ is divergence-free, we have 
\[
 -p\io \frac{\nep^{p-1}\uep\cdot \na \nep}{(\eta-\cep)^\theta} - \theta \io \frac{\nep^p \uep\cdot \na \cep}{(\eta-\cep)^{1+\theta}} = - \io \uep\cdot \na \kl{\frac{\nep^p}{(\eta-\cep)^\theta}} = 0 \qquad \mbox{ on } (T,\infty).
\]
Furthermore, employing H\"older's inequality, we estimate 
\[
 \io \frac{\nep^p}{(\eta-\cep)^\theta} \leq \left(\io \frac{1^{p+1}}{(\eta-\cep)^\theta}\right)^{\frac1{p+1}} \left(\io \frac{\nep^{p+1}}{(\eta-\cep)^\theta}\right)^{\frac p{p+1}} \leq \left(\frac{2^\theta|\Om|}{\eta^\theta}\right)^{\frac1{p+1}} \left(\io\frac{\nep^{p+1}}{(\eta-\cep)^\theta}\right)^{\frac{p}{p+1}} \qquad \mbox{on } (T,\infty)
\]
by \eqref{eq:csmall}, 
so that 
\[
 -\mu p\io \frac{\nep^{p+1}}{(\eta-\cep)^\theta}\leq -p\mu \left(\frac{\eta^\theta}{2^\theta|\Om|}\right)^{\frac1p} \left(\io \frac{\nep^p}{(η-\cep)^\theta}\right)^{1+\frac1p}=:-k_1 y_{ε}^{1+\frac1p} \qquad \mbox{on } (T,\infty).
\]

We infer from \eqref{eq:ddtionpetactheta} by integration by parts that 
\begin{align*}
 y_{ε}'&\leq  -p(p-1)\io \frac{\nep^{p-2}|\na \nep|^2}{(\eta-\cep)^\theta} - p\theta \io \frac{\nep^{p-1}\na \nep•\na \cep}{(\eta-\cep)^{\theta+1}} + \chi p(p-1)\io \frac{\nep^{p-1}\na \nep\cdot \na \cep}{(1+\eps \nep)(\eta-\cep)^\theta}\\ 
 &+\chi p\theta \io \frac{\nep^p|\na \cep|^2}{(1+\eps \nep)(\eta-\cep)^{1+\theta}} + \kappa p y_{ε} - k_1  y_{ε}^{1+\frac1p}\\
 & -\theta p\io \frac{\nep^{p-1}\na \nep\cdot \na \cep}{(\eta-\cep)^{1+\theta} } - \theta (1+\theta)\io \frac{\nep^p|\na \cep|^2}{(\eta-\cep)^{2+\theta}} - \theta \io \frac{\nep^p\frac \cep\eps\ln(1+\eps \nep)}{(\eta-\cep)^{1+\theta}} \qquad \mbox{ on } (T,\infty).
\end{align*}
 
We use that $\frac{1}{1+\eps \nep}\leq 1$ and (by \eqref{eq:csmall}) $1\leq \frac{\eta}{\eta-\cep}\leq 2$ as well as nonpositivity of the last term and get
\begin{align*}
 y_{ε}'\leq& -p(p-1)\io \frac{\nep^{p-2}|\na \nep|^2}{(\eta-\cep)^\theta} - [\theta(1+\theta)-\chi p\theta\eta] \io \frac{\nep^p|\na \cep|^2}{(\eta-\cep)^{2+\theta}}\\
 &+[p\theta+\chi p(p-1)\eta +\theta p] \io \frac{\nep^{p-1}|\na \nep||\na \cep|}{(\eta-\cep)^{1+\theta}} + \kappa p y_{ε} - k_1 y_{ε}^{1+\frac1p}, \qquad \mbox{ on } (T,\infty),
\end{align*}
where an application of Young's inequality reveals that for any $t>T$
\begin{align*}
 [2p\theta&+χp(p-1)η]\io \frac{\nep^{p-1}(t)|\na\nep(t)||\na\cep(t)|}{(η-\cep(t))^{1+\theta}}\\
 &\leq \frac{(2p\theta+χp(p-1)η)^2}{4\theta(1+\theta)-\chi p\thetaη} \io \frac{\nep^{p-2}(t)|\na\nep(t)|^2}{(η-\cep)^{1+\theta}}+(\theta(1+\theta)-χp\thetaη)\io\frac{\nep^p(t)|\na\cep(t)|^2}{(η-\cep)^{2+\theta}}
\end{align*}
and thereby leads to 
\begin{align}
 y_{ε}' \leq& \left(\frac{[2p\theta+\chi p(p-1)\eta]^2}{4(\theta(1+\theta)-\chi p \theta \eta)} - p(p-1) \right)\io \frac{\nep^{p-2}|\na \nep|^2}{(\eta-\cep)^\theta} + \kappa p y_{ε} -k_1 y_{ε} ^{1+\frac1p}\nn\\
 \leq& \kappa p y_{ε} -k_1y_{ε}^{1+\frac1p} \qquad \mbox{on } (T,\infty)\label{eq:comparewithearlier}
\end{align}
by \eqref{eq:reasonforchoicethetaeta}. 
Because 
\begin{align}\label{eq:uppersoln}
 z(t):=&\left[\left(\frac{1}{y(T)^{\frac1p}}-\frac{k_1}{\kappa p}\right)e^{-\kappa(t-T)}+\frac{k_1}{\kappa p}\right]^{-p}
 \leq \left[\frac{k_1}{\kappa p}\left(1-e^{-\kappa(t-T)}\right)\right]^{-p}, \qquad t>T,
\end{align}
solves $z'=\kappa p z-k_1 z^{1+\frac1p}$, $z(T)=y_{ε}(T)$, by the usual ODE comparison argument we infer $y_{ε}\leq z$ on $(T,\infty)$ and thus 
\[
 \io \nep^p\leq \io (\eta-\cep)^\theta \frac{\nep^p}{(\eta-\cep)^\theta}\leq \eta^\theta\io\frac{\nep^p}{(\eta-\cep)^\theta} = \eta^\theta y_{ε}(t) \leq \eta^\theta\left[\frac{k_1}{\kappa p}(1-e^{-\kappa})\right]^{-p}.
\]
for $t>T^\star:=T+1$.
\end{proof}

One particular consequence of this bound is the following:

\begin{lemma}\label{lem:intintnp}
 For any $p>1$ and any $δ>0$ there is $T>0$ such that for any $t>T$ there is $ε_0>0$ such that for any $ε∈(0,ε_0)$ 
\[
 \int_t^{t+3} \norm[\Lom p]{\nep-\frac{κ}{μ}} < δ.
\]
\end{lemma}
\begin{proof}
 Let $δ>0$, let $p>1$. Employing Lemma \ref{lem:nepbd} we let $T_0>0$, $ε_\star>0$ and $C>0$ be such that $\io \nep^{2p}(t)\leq C^{2p}$ for all $t>T_0$, $ε\in(0,ε_\star)$. 
 From \eqref{eq:iinminusgwqleq} and \eqref{conv:nl2} we infer that $\intninf\!\! \io \left(n-\frac{κ}{μ}\right)^2$ is finite. Hence, there is $T>T_0$ such that for all $t>T$ we have $\int_t^{t+3} \io \left(n-\frac{κ}{μ}\right)^2 \leq \frac12 δ^{2p-2} 3^{3-2p} (C+\frac{κ}{μ}|\Om|^{\frac1{2p}})^{4-2p}$. Due to \eqref{conv:nl2}, for all $t>T$ we can find $ε_t>0$ such that for all $ε∈(0,ε_t)$ we have $\int_t^{t+3}\io \left(\nep-\frac{κ}{μ}|\Om|^{\frac{1}{2p}}\right)^2<δ^{2p-2} 3^{3-2p} \kl{C+\frac{κ}{μ}|\Om|^{\frac1{2p}}}^{4-2p}$. 
 For any $t>T$, we let $ε_0:=\min\set{\eps_\star,ε_t}$. 
 By interpolation and Hölder's inequality, for any $t>T$ and any $ε∈(0,ε_0)$: 
\begin{align*}
 \int_t^{t+3}\norm[\Lom p]{\nep-\frac{κ}{μ}} &\leq \int_t^{t+3} \norm[\Lom 2]{\nep-\frac{κ}{μ}}^{\frac1{p-1}}\norm[\Lom {2p}]{\nep-\frac{κ}{μ}}^{\frac{p-2}{p-1}} \\
 &\leq  \left(\int_t^{t+3}\norm[\Lom 2]{\nep-\frac{κ}{μ}}^2\right)^{\frac1{2p-2}} \left(\int_t^{t+3}\norm[\Lom{2p}]{\nep-\frac{κ}{μ}}^{\frac{2p-4}{2p-3}}\right)^{\frac{2p-3}{2p-2}} \\
&\leq \left(δ^{2p-2}3^{3-2p} \left(C+\frac{κ}{μ}|\Om|^{\frac1{2p}}\right)^{4-2p}\right)^{\frac1{2p-2}} \left(3 \left(C+\frac{κ}{μ}|\Om|^{\frac1{2p}}\right)^{\frac {2p-4}{2p-3}}\right)^{\frac{2p-3}{2p-2}}=δ\qedhere 
\end{align*}
\end{proof}

\subsection{Convergence of $u$}

As starting point for convergence and eventual smoothness of $u$ we prove the following 

\begin{lemma}\label{lem:ul6}
 For any $q\in[1,6)$ and any $\eta>0$ there is $T>0$ such that for any $t>T$ one can find $\eps_0>0$ such that for any $\eps\in(0,\eps_0)$ 
\[
 \intttpe\norm[\Lom q]{\uep}^2<\eta.
\]
\end{lemma}
\begin{proof}
According to Lemma \ref{lem:ddtiou2} applied to $ζ=\frac{κ}{μ}$, there is $C>0$ such that for any $ε>0$
\[
 \ddt\io |\uep|^2 + \io |\na\uep|^2 \leq C \io \left(\nep-\frac{κ}{μ}\right)^2 + C \left(\io |f|^{\frac65}\right)^{\frac 53}\qquad \text{ on } (0,∞).
\]
Due to \eqref{reg:f} and the uniform bound for $\intninf\!\! \io \left(\nep-\frac{κ}{μ}\right)^2$ from \eqref{eq:iinminusgwqleq}, apparently there is $C>0$ such that 
\[
 \io |\uep(t)|^2-\io u_0^2+\intnt \io |\na \uep|^2 \leq C 
\]
for all \epscond\ and any $t>0$.
Accordingly, due to \eqref{conv:nau}, 
\[
 \intninf\!\! \io |\na u|^2 \leq \io u_0^2 + C \quad \mbox{ and hence } \quad \lim_{t\to∞} \int_t^{t+1} \io |\na u|^2 = 0.
\]
Because $W^{1,2}(\Om)\embeddedinto \Lom q$, for any $\eta>0$ there is $T>0$ such that for any $t>T$ we have $\intttpe\norm[\Lom q]{u}^2 <\frac{\eta}2$ and thus, by \eqref{conv:u}, for any $\eta>0$ there is $T>0$ such that for any $t>T$ there is $\eps_0>0$ such that for any $\eps\in(0,\eps_0)$, \(\intttpe\norm[\Lom q]{\uep}^2<η\).
\end{proof}

\begin{lemma}\label{lem:ueptozero}
 For any $p\in [6,\infty)$ and any $δ>0$ there is $T>0$  such that for any $t>T$ there is \(ε_0>0\) such that for any \(ε∈(0,ε_0)\) 
\[
 \norm[\Lom p]{\uep(\cdot,s)}<δ \qquad \mbox{for any } s\in [t,t+1].
\]
\end{lemma}
\begin{proof}
 We let $p\ge 6$ and choose $q\in(3,6)$ such that 
 \begin{equation}\label{eq:chooseq}
  \frac12-\frac{3}{2p}-3\kl{\frac1q-\frac1p} = \frac12+\frac{3}{2p} -\frac3q\geq 0
 \end{equation}
 and define $γ:=\frac32(\frac1q-\frac1p)$. We use $L^p$-$L^q$-estimates for the Stokes semigroup (see e.g. \cite[Lemma 2.3]{XinruJoh}) to choose constants $k_1$, $k_2$, $k_3$, $k_4$, $k_5$ such that  
 \begin{align}\label{eq:chooseCStokesdecay}
  \norm[\Lom p]{e^{-tA}\calP ϕ} \leq& k_1 t^{-γ} \norm[\Lom  q]{ϕ}\qquad &&\text{for all }\phi\in \Lom q \mbox{  and all } t>0 \nn\\
  \norm[\Lom p]{e^{-tA}\calP\na\cdot ϕ} \leq& k_2 t^{-\frac12-\frac3{2p}}\norm[\Lom {\frac p2}]{ϕ}\qquad &&\text{for all }\phi\in \Lom{\frac p2}\mbox{  and all } t>0 \nn\\
  \norm[\Lom p]{e^{-tA}\calP ϕ} \leq& k_3 \norm[\Lom p]{ϕ}\qquad &&\text{for all }\phi\in \Lom p\mbox{  and all } t>0 \\
  \int_0^3 \norm[\Lom p]{e^{-(3-s)A} \calP ϕ} \leq& \rlap{$\displaystyle k_4 \int_0^3 (3-s)^{-\frac32(\frac23-\frac1p)} \norm[\Lom{\frac32}]{\calP ϕ(s)} \leq k_5 \sup_{s\in(0,3)} \norm[\Lom {\frac32}]{ϕ(s)}$}\nn \\
& &&\text{ for all } ϕ\in L^∞((0,3),L^{\frac32}(\Om)) \nn
 \end{align}
 and pick $δ_0∈(0,δ)$ such that 
\begin{equation}\label{eq:choosedelta}
 δ_0<\left(2k_2 \int_0^3 s^{-\frac12-\frac3{2p}} s^{-2γ}ds\right)^{-1}\quad \mbox{ and } \quad \delta_0\leq\kl{2k_2\int_0^3 s^{-\frac12-\frac3{2p}}ds}^{-1}
\end{equation}

 We then pick $t_0$ such that for every $t>t_0$ we can find $ε_t>0$ such that for any \(ε∈(0,ε_t)\)
\begin{align*}
 \intttpe \norm[\Lom q]{\uep} <& \frac{δ_0}{4k_1}, & 
 \int_t^{t+3}\norm[\Lom p]{\nep-\frac{κ}{μ}}<& \frac{δ_0}{3^\gamma 4 k_3\norm[\Lom{∞}]{\na Φ}}, 
 \sup_{s\in(t,t+3)}\norm[\Lom {\frac32}]{f(s)}<&\frac{δ_0}{3^{γ}4k_5},
\end{align*}
which is possible due to Lemma \ref{lem:ul6} (applied to $η=\kl{\frac{\delta_0}{4k_1}}^2$ and combined with Hölder's inequality), Lemma \ref{lem:intintnp} and \eqref{reg:f}. 
We let $t_1>t_0$ and $ε∈(0,ε_{t_1})$ and find $t_\star∈(t_1,t_1+1)$ such that $\norm[\Lom q]{\uep(t_\star)}<\frac{δ_0}{4k_1}$. 
We define 
\[
 T=t_0+2.
\]

In $X=\set{ v\colon \Om\times(t_\star,t_\star+3) \to ℝ;\; \sup_{s\in(0,3)} s^\gamma \norm[\Lom p]{v(t_\star+s)}\le δ_0}$ we now consider the mapping $Ψ\colon X\to X$ given by  
\[
 Ψ(v) =e^{-tA}u(t_\star) + \int_{t_\star}^{t} e^{(t-s)A} \calP\left[-\na\cdot (\Yep v\otimes v)(s) + \nep(s)\naΦ + f(s)\right] ds.
\]
First, we verify that actually $Ψ(v)\in X$ for all $v\in X$. 
Taking into account \eqref{eq:chooseCStokesdecay}, for any such $v$ we may estimate 
\begin{align*}
 \norm[\Lom p]{Ψv(t)} \leq& k_1(t-t_\star)^{-γ} \norm[\Lom q]{\uep(t_\star)} + k_2\int_{t_\star}^{t}(t-s)^{-\frac12-\frac32(\frac2p-\frac1p)} \norm[\Lom{\frac p2}]{v\otimes v}\\
& + k_3\norm[\Liom]{\naΦ} \int_{t_\star}^{t_\star+3} \norm[\Lom p]{\nep-\frac{κ}{μ}} + k_5 \sup_{s\in(t_\star,t_\star+3)} \norm[\Lom {\frac32}]{f}
\end{align*}
for all $t\in(t_\star,t_\star+3)$. 
Thus, if we use the choice of $t_\star$ and $ε$, that $\Yep$ is contracting and that $\norm[\Lom{\frac p2}]{v\otimes v}\leq\norm[\Lom p]{v}^2$ by Hölder's inequality, we see that for every $t\in(t_{\star},t_{\star}+3)$ and every $v\in X$
\begin{align}\label{eq:ueplpsmall}
  (t-t_\star)^{γ}\norm[\Lom p]{Ψ(v)(t)} &\leq \frac{δ_0}4 + δ_0\left(δ_03^{γ}k_2 \int_{t_\star}^t (t-s)^{-\frac12-\frac3{2p}} (s-t_\star)^{-2γ} ds \right) +\frac{δ_0}4+\frac{δ_0}4 \leq δ_0,
\end{align} 
where for the estimate $\int_{t_\star}^t (t-s)^{-\frac12-\frac3{2p}} (s-t_\star)^{-2γ} ds\leq \int_0^3 (t-s)^{-\frac12-\frac3{2p}} s^{-2γ}ds $ we rely on \eqref{eq:chooseq} and where we take into account \eqref{eq:choosedelta}. 
Moreover, for any $v,w\in X$, 
\begin{align*}
 \norm[\Lom{\frac p2}]{v\otimes v - w\otimes w}&=\norm[\Lom {\frac p2}]{v\otimes(v-w)+(v-w)\otimes w}\\
  &\leq (\norm[\Lom p]{v}+\norm[\Lom p]{w})\norm[\Lom p]{v-w}\leq 2δ_0\norm[\Lom p]{v-w}
\end{align*}
and hence 
\begin{align*}
 \norm[\Lom p]{Ψ(v)(t)-Ψ(w)(t)}&\leq k_2\int_{t_\star}^{t} (t-s)^{-\frac12-\frac3{2p}} \norm[\Lom {\frac p2}]{v\otimes v - w\otimes w} ds \\
 &\leq 2k_2δ_0\int_0^3 s^{-\frac12-\frac3{2p}} ds \norm[\LT\infty{\Lom p}]{v-w},
\end{align*}
so that $Ψ$ apparently is a contraction on $X$. Therefore, there is a unique fixed-point of $Ψ$ on $X$, which, due to the definition of $Ψ$, must coincide with the unique weak solution $\uep$ of \eqref{eq:uep} on $(t_\star,t_\star+3)$ (cf. \cite[Thm. V.2.5.1]{sohr_book}). 
From \eqref{eq:ueplpsmall} we may conclude that 
\[
 \norm[\Lom p]{\uep(\cdot,t)} \leq δ
\]
for all $t\in(t_\star+1,t_\star+3)\supset(t_1+2,t_1+3)$.
\end{proof}

In the following lemmata, we will attempt to prove Hölder regularity of the components of a solution on intervals of the form $(t_0,t_0+1)$ for $t_0>0$ by using that they satisfy certain PDEs. The estimates used for this purpose take into account initial data, that is, e.g., $u(t_0)$, about which we do not know much. Therefore, we introduce the following cut-off functions:

\begin{definition}\label{def:xi} 
Let $ξ_0\colonℝ\to[0,1]$ be a smooth, monotone function, satisfying $ξ_0\equiv 0$ on $(-∞,0]$ and $ξ_0\equiv 1$ on $(1,∞)$ and for any $t_0∈ℝ$ we let 
$ξ_{t_0}:=ξ_0(\cdot-t_0)$. 
\end{definition}

We will employ this function in the proof of the following lemma on regularity of $\uep$.

\begin{lemma}\label{lem:uepc1alpha}
There are $α∈(0,1)$, $T>0$ and $C>0$ such that for any $t>T$ one can find $ε_0>0$ such that for any $ε\in(0,ε_0)$ the estimate 
 \begin{equation}\label{eq:lemholderestuep}
 \nnorm{C^{1+\alpha,\frac{α}2}(\Ombar\times[t,t+1])}{\uep}\le C
\end{equation}
 holds true.
\end{lemma}
\begin{proof}
Let $s>3$ and $s_1>2s$. Let $r>1$ and let $s_1'$ be  such that $\frac1{s_1}+\frac1{s_1'}=1$. 
According to Lemma \ref{lem:nepbd} there are $C>0$, $T_1>0\) satisfying that for any $t>T_1$ there is $ε_1>0$ such that 
\begin{equation}\label{eq:uepcalpha-nepls}
 \intttpe \norm[\Lom s]{\nep}<C \qquad \mbox{for all }  ε\in(0,ε_1)
\end{equation}
and Lemma \ref{lem:ueptozero} makes it possible to find $T\ge T_1$, such that for all $t>T$ there is $ε_t∈(0,ε_1)$ such that 
\begin{equation}\label{eq:uepcalpha-ubds}
 \norm[L^\infty((t,t+2),L^r(\Om))]{\uep}<C, \qquad \norm[L^\infty((t,t+2),L^s(\Om))]{\uep}<C, \qquad \norm[L^\infty((t,t+2),L^{s_1'}(\Om))]{\uep}<C
\end{equation}
for every $\eps∈(0,ε_t)$.
Moreover, given any $t_0>T$ we let $ξ:=ξ_{t_0}$ as in Definition \ref{def:xi} and note that due to \eqref{eq:uepcalpha-nepls}, \eqref{eq:uepcalpha-ubds} and \eqref{reg:f} there is $k_0>0$ such that 
\begin{equation}\label{eq:intest949}
 \intttz \norm[\Lom s]{\calPξ\left(n-\frac{κ}{μ}\right)\naΦ}^s+\intttz \norm[\Lom s]{\calP ξ'\uep}^s + \intttz \norm[\Lom s]{\calP ξf}^s\leq k_0
\end{equation}
for any $t_0>T$ and $ε∈(0,ε_{t_0})$. 
We then for any $t_0>0$ let $ξ:=ξ_{t_0}$ and observe that the function $ξ\uep$ solves 
\begin{align*}
 (ξ\uep)_t =& \Delta (ξ\uep) - (\Yep\uep•\na)ξ\uep + \na (ξP_{\eps}) + ξ\nep\naΦ +ξf+ξ'\uep \qquad \mbox{in } \Om\times(t_0,∞)\\
 \na\cdot(ξ\uep)=&0 \qquad \mbox{in } \Om\times(t_0,∞)\\
 (ξ\uep)(\cdot,t_0)=&0 \quad \mbox{ in } \Om,\qquad (ξ\uep)=0 \quad \mbox{ on } \dOm\times(t_0,∞)
\end{align*}
and hence the known maximal Sobolev regularity estimate for the Stokes semigroup (\cite{giga_sohr}) yields a constant $k_1>0$ such that 
\begin{align}\label{eq:sgestu}
 \intttz& \norm[\Lom s]{(ξ\uep)_t}^s+\intttz \norm[\Lom s]{D^2(ξ\uep)}^s\\\ 
&\leq k_1\bigg[0+ \intttz \norm[\Lom s]{\calP((ξ\Yep\uep•\na)\uep)}^s +\intttz \norm[\Lom s]{\calPξ\left(n-\frac{κ}{μ}\right)\naΦ}^s\nn
\\&+\qquad\quad\intttz \norm[\Lom s]{\calP ξ'\uep}^s + \intttz \norm[\Lom s]{\calP ξf}^s\bigg]\nn
\end{align}

From the boundedness of the Helmholtz projection in $L^r$-spaces and Hölder's inequality we obtain $k_2>0$ such that for any $t_0>T$
\begin{align*}
 \norm[\Lom s]{\calP (\Yep\uep•\na)ξ\uep(t)}^s &\leq k_2\norm[\Lom {s_1'}]{\Yep\uep(t)}^s\norm[\Lom{s_1}]{\na(ξ\uep(t))}^s \\
  & \leq k_2\norm[\Lom{s_1'}]{\uep(t)}^s\norm[\Lom {s_1}]{\na(ξ\uep)(t)}^s \leq k_2 C^s\norm[\Lom {s_1}]{\na(ξ\uep)(t)}^s
\end{align*}
for all $t\in(t_0,t_0+2)$ and any $ε\in(0,ε_{t_0})$.
We let $a=\frac{\frac13-\frac1{s_1}+\frac1r}{\frac23-\frac1s+\frac1r}$ and observe that $a\in(\frac12,1)$. 
Hence the \GNI\ provides us with $k_3>0$ such that for all $t_0>T$
\begin{align*}
 \norm[\Lom {s_1}]{\na(ξ\uep(t))}^s&\leq  k_3\norm[\Lom s]{D^2(ξ\uep)(t)}^{as}\norm[\Lom r]{ξ\uep(t)}^{(1-a)s}\\
  &\leq k_3C^{(1-a)s} \norm[\Lom s]{D^2(ξ\uep)(t)}^{as}\qquad \mbox{for all } t∈(t_0,t_0+2), ε∈(0,ε_{t_0}).
\end{align*}
Therefore, employing \YI, we find $k_4>0$ such that for all $t_0>T$ 
\begin{align*}
 \int_{t_0}^{t_0+2} \norm[\Lom s]{\calP(\Yep\uep\naξ\uep)}^s\leq k_1k_2k_3\int_{t_0}^{t_0+2} \norm[\Lom s]{D^2(ξu)}^{as} 
 \leq \frac12\int_{t_0}^{t_0+2}\norm[\Lom s]{D^2(ξu)}^s + 2k_4 
\end{align*}
for all $ε∈(0,ε_{t_0})$. 
Combining this with \eqref{eq:intest949}, we thus can find $k_5>0$ such that \eqref{eq:sgestu} turns into 
\[
 \int_{t_0}^{t_0+2}\norm[\Lom s]{(ξ_{t_0}\uep)_t}^s+\frac12\int_{t_0}^{t_0+2} \norm[\Lom s]{D^2(ξ_{t_0}\uep)} \leq k_5 
\]
for any $t_0>T$ and any $ε∈(0,ε_0)$.
Accordingly, for any $s>1$ there are $C>0$, $T>0$ such that for any $t>T$ there is $ε_0>0$ satisfying that for any $ε\in(0,ε_0)$ 
\[
 \norm[L^s((t,t+1),L^s(\Om))]{\uept}+\norm[L^s((t,t+1),W^{2,s}(\Om))]{\uep}\leq C.
\]
Finally, by \cite[Thm. 1.1]{amann_cptembeddings}, this implies \eqref{eq:lemholderestuep}.
\end{proof}

\subsection{Eventual smoothness of $c$}
Applying a similar reasoning, concerning $c$ we obtain bounds of the same kind. 

\begin{lemma}\label{lem:cc1pa}
 For every $p∈(1,∞)$ there are $C>0$, $T>0$, $α∈(0,1)$ such that for any $t>T$ there is  $ε_0>0$ such that
\begin{equation}\label{eq:csobest}
 \intttpe \norm[\Lom p]{\cept} +\intttpe \norm[W^{2,p}(\Om)]{\cep} \leq C
\end{equation}
 for all $ε∈(0,ε_0)$ and moreover there are $C>0$, $T>0$, $α∈(0,1)$ such that for any $t>T$ there is  $ε_0>0$ such that
\[
 \nnorm{C^{1+γ,\frac{γ}2}(\Ombar\times[t,t+1])}{\cep} \leq C
\]
 for any $ε∈(0,ε_0)$.
\end{lemma}
\begin{proof}
 Let $p∈(1,\infty)$ and choose $q∈(1,p)$. 
 We use Lemma \ref{lem:uepc1alpha} and Lemma \ref{lem:nepbd} to choose $T_\star>0$ such that there are $C_u>0$, $C_n>0$, for any $t>T_\star$ allowing us to find $ε_t>0$ such that 
\[
 \norm[L^\infty(\Om\times(t,t+2))]{\uep}\leq C_u \quad \mbox{and} \quad \norm[L^\infty((T,\infty),L^p(\Om))]{\nep}\leq C_n
\]
 for all $ε∈(0,ε_t)$. We then employ maximal Sobolev estimates for the Neumann heat semigroup (\cite{giga_sohr}), which yield $k_1>0$ such that 
\begin{align*}
 \int_{t_0}^{t_0+2} &\norm[\Lom p]{(ξ\cep)_t}^p+\int_{t_0}^{t_0+2}\norm[\Lom p]{\Delta (ξ\cep)}^p \\& \leq k_1\left( 0+ \int_{t_0}^{t_0+2} \norm[\Lom p]{ξ\uep•\na\cep}^p+\int_{t_0}^{t_0+2}\norm[\Lom p]{ξ\cep\ln(1+ε\nep)}^p+\int_{t_0}^{t_0+2}\norm[\Lom p]{ξ'\cep}^p\right)\\
 &\leq k_1C_u^p\int_{t_0}^{t_0+2} \norm[\Lom p]{\na(ξ\cep)}^p+ 2k_1\norm[\Liom]{c_0}^p(C_n^p+\norm[L^\infty(ℝ)]{ξ'})
\end{align*}
for any $t_0>T$ and any $ε∈(0,ε_0)$.
With $k_2>0$ being the constant featured by the \GNI\ (aided by e.g. \cite[Thm. 3.4]{Simader_weakDirichletandNeumann} for replacing $D^2$ by $\Delta$), for any $t_0>T$ we moreover have 
\begin{align*}
 \norm[\Lom p]{\na(ξ\cep)(t)}^p \leq& k_2 \norm[\Lom p]{\Delta (ξ\cep)(t)}^{ap}\norm[\Lom q]{ξ\cep(t)}^{(1-a)p} + k_2\norm[\Liom]{ξ\cep(t)}^p\\
 \leq& \frac12 \norm[\Lom p]{\Delta (ξ\cep)(t)}^p + k_3 \norm[\Lom q]{ξ\cep(t)}^p + k_2\norm[\Liom]{c_0}^p \  \mbox{for any }ε∈(0,ε_{t_0}), t∈(t_0,t_0+2),
\end{align*}
where $k_3>0$ is obtained from \YI\ and $a:= \frac{\frac13-\frac1p+\frac1q}{\frac23-\frac1p+\frac1q}$ satisfies $a∈(\frac12,1)$.
In total, we have found $k_4>0$ such that for any $t_0>T$ 
\[
 \int_{t_0}^{t_0+2} \norm[\Lom p]{(ξ\cep)_t}^p + \frac12 \int_{t_0}^{t_0+2}\norm[\Lom p]{\Delta (ξ\cep)}^p \leq k_4
\]
holds for any $ε\in(0,ε_{t_0})$. Due to $ξ\equiv 1$ on $(t_0+1,t_0+2)$, we in particular have shown that for any $p>1$ there are $C>0$, $T:=T_\star+1>0$ such that for any $t_0>T$ we can find $ε_0>0$ such that for any $ε∈(0,ε_0)$
\[
 \int_{t_0+1}^{t_0+2}\norm[\Lom p]{\cept}+\int_{t_0+1}^{t_0+2} \norm[W^{2,p}(\Om)]{\cep} \leq C
\]
Using sufficiently high values of $p$, an application of the embedding result in \cite{amann_cptembeddings} refines this into the assertion on Hölder continuity.
\end{proof}

\subsection{Smoothness of $n$}

Inter alia depending on Lemma \ref{lem:nepbd} and \eqref{eq:csobest}, we can achieve the same for $\nep$: 
\begin{lemma}\label{lem:nc1pa}
 There are $\alpha\in(0,1)$, $C>0$, $T>0$ such that for any $t>T$ there is $ε_0>0$ with any $ε\in(0,ε_0)$ satisfying 
\[
 \nnorm{C^{1+γ,\frac{γ}2}(\Ombar\times[t,t+1])}{\nep}\leq C.
\]
\end{lemma}
\begin{proof}

Let $p∈(1,∞)$, $q∈(1,p)$. Using Lemma \ref{lem:nepbd}, we fix $T_0>0$, $C_n>0$ such that for any $ε>0$, for any $t>T_0$ 
\[
 \norm[\Lom {2p}]{\nep(•,t)}\leq C_n\qquad \norm[\Lom q]{\nep(•,t)}\leq C_n,\qquad \norm[\Lom p]{\nep(•,t)}\leq C_n. 
\]
Aided by Lemma \ref{lem:cc1pa} and Lemma \ref{lem:ueptozero} we then choose $C_c>0$, $C_u>0$ and $T>T_0$ such that for any $t>T$ there is $ε_t>0$ such that for any $ε\in(0,ε_t)$ 
\begin{align*}
 \norm[L^\infty(\Om\times(t,t+2))]{\na \cep}\leq C_c\qquad \int_t^{t+2} \norm[\Lom{2p}]{\Delta \cep}^{2p}\leq C_c\\
 \norm[\Liom]{\uep(•,s)}\leq C_u \qquad \mbox{for all } s\in (t,t+2).
\end{align*}

From the \GNI\ (combined with e.g. \cite[Thm. 3.4]{Simader_weakDirichletandNeumann} for estimating $\norm{D^2ϕ}$ by $\norm{\Delta ϕ}$) and \YI\ we see that for any $η>0$ we can find $C_η>0$ such that 
\begin{align}\label{eq:estimatenabla}
 \norm[\Lom p]{\na ϕ}^p \leq k_1 \norm[\Lom p]{\Delta ϕ}^{ap} \norm[\Lom q]{ϕ}^{(1-a)p} + k_1 \norm[\Lom q]{ϕ}^p\leq η\norm[\Lom p]{\Delta ϕ}^p + C_η\norm[\Lom q]{ϕ}^p \quad \mbox{for all }ϕ∈ W^{2,p}(\Om),
\end{align}
where $a=\frac{\frac13+\frac1q-\frac1p}{\frac23+\frac1q-\frac1p}$ and $k_1>0$ is the constant obtained from the \GNI. For any $t>T$ and any $ε\in(0,ε_t)$ we may estimate 
\begin{align*}
 \norm[\Lom p]{ξ\na•\kl{\frac{\nep}{1+ε\nep}\na\cep}}^p &= \norm[\Lom p]{ξ\frac{\na\nep•\na\cep}{(1+ε\nep)^2}+ξ\frac{\nep}{1+ε\nep}\Delta\cep}^p\\
&\leq 2^p\norm[\Lom p]{\na(ξ\nep)•\na\cep}^p+ 2^p\norm[\Lom{2p}]{\nep}^{p}\norm[\Lom {2p}]{\Delta \cep}^p \\
&\leq 2^p C_c^p \norm[\Lom p]{\na(ξ\nep)}^p + 2^p C_n^p \norm[\Lom{2p}]{\Delta \cep}^p
\end{align*}
on $(t,t+2)$. 
The maximal Sobolev estimates for the heat semigroup (\cite{giga_sohr}) once more assert that with some $k_2>0$
\begin{align*}
 \int_t^{t+2}& \norm[\Lom p]{ξ\nept}^p + \int_t^{t+2}  \norm[\Lom p]{\Delta ξ\nep}^p \\
 &\leq k_2 \int_t^{t+2}  \norm[\Lom p]{ξ\uep•\na\nep}^p +k_2\chi \int_t^{t+2}  \norm[\Lom p]{ξ\na •\left(\frac{\nep}{1+ε\nep}\na\cep\right)}^p\\
  & + k_2\kappa \int_t^{t+2}  \norm[\Lom p]{ξ\nep}^p + k_2μ \int_t^{t+2}  \norm[\Lom p]{ξ\nep^2}^p\\ 
\intertext{for any $t>0$ and any $ε>0$. Taking into account the estimates prepared above, for $t>T$ and $ε\in(0,ε_t)$ we thus obtain}
 \int_t^{t+2}& \norm[\Lom p]{ξ\nept}^p + \int_t^{t+2}  \norm[\Lom p]{\Delta (ξ\nep)}^p \\
 &\leq k_2(C_u^p+2^p C_c^p)\int_t^{t+2}  \norm[\Lom p]{\na(ξ\nep)}^p + 2^pC_n^p\int_t^{t+2}  \norm[\Lom{2p}]{\Delta\cep}^p + k_2(κ+μ)C_n(t+1-t)\\
 &\leq 2^pC_n^p\sqrt{C_c}+c(κ+μ)C_n(t+1-t) + \frac12\int_t^{t+2}  \norm[\Lom p]{\Delta (ξ\nep)} + C_η\norm[\Lom q]{\nep}^p,
\end{align*}
where we have used \eqref{eq:estimatenabla} with $η=\frac1{2k_2(C_u^p+2^pC_c^p)}$. Finally, since $ξ\equiv 1$ on $(t+1,t+2)$, we conclude
\[
 \int_{t+1}^{t+2} \norm[\Lom p]{\nept}^p + \frac12 \int_{t+1}^{t+2} \norm[\Lom p]{\Delta \nep}^p \leq C_1
\]
for any $t>T$ and any $ε\in(0,ε_t)$, where we have set $C_1= 2^pC_n^p\sqrt{C_c}+k_2(κ+μ)C_n(t+1-t)+C_η C_n$.
\end{proof}

\subsection{Improved smoothness}

Having found uniform Hölder bounds on $\nep$, $\cep$, $\uep$ for $\eps>0$ in the previous three lemmata, also $n$, $c$ and $u$ share this regularity and these bounds.

\begin{corollary}\label{cor:convc1pa}
There are $γ\in(0,1)$ and $T_0>0$ as well as a subsequence $\{\eps_{j_k}\}_{k\nat}\downto0$ of the sequence from Proposition \ref{prop:ex-weaksol} such that  for any $t>T_0$ 
 \[
  \nep\to n,\quad \cep\to c \quad \mbox{ in } \quad C^{1+γ,\frac{γ}2}(\Ombar\times[t,t+1]),\qquad \uep\to u \quad \mbox{ in } \quad C^{1+γ,\frac{γ}2}(\Ombar\times[t,t+1])
 \]
 as $\eps=\eps_j\downto 0$.
Moreover, there is $C>0$ such that 
\begin{equation}\label{eq:holderest}
 \nnorm{C^{1+γ,\frac{γ}2}(\Ombar\times[t,t+1])}{n}\leq C, \quad \nnorm{C^{1+γ,\frac{γ}2}(\Ombar\times[t,t+1])}{c}\leq C,\quad \nnorm{C^{1+γ,\frac{γ}2}(\Ombar\times[t,t+1])}{u}\leq C
\end{equation}
\end{corollary}
\begin{proof}
This is an immediate consequence of Lemma \ref{lem:nc1pa}, Lemma \ref{lem:cc1pa} and Lemma \ref{lem:uepc1alpha}.
\end{proof}

\begin{lemma}\label{lem:c2alpha}
 There are $T>0$, $γ∈(0,1)$ such that 
\[
 n,c\in C^{2+γ,1+\frac{γ}2}(\Ombar\times[T,\infty)), \quad u∈ C^{2+γ,1+\frac{γ}2}(\Ombar\times[T,\infty)).
\]
\end{lemma}
\begin{proof}
With $ξ$ as in Definition \ref{def:xi}, (and with $T$ as in the previous lemmata) the problem 
\begin{equation}
 \ctilde_t=\Delta \ctilde + g,\quad \ctilde(T)=0,\quad \delny\ctilde\amrand=0,
\end{equation}
for $g=-ξnc-ξu\na c+cξ'\in C^γ(\Ombar\times(T,\infty))$ is solved by $ξc$, the solution of this problem is unique according to \cite[III.5.1]{LSU}, and there is a solution belonging to $C^{2+γ,1+\frac{γ}2}(\Ombar\times[T+1,∞))$. We conclude that $c\in C^{2+γ,1+\frac{γ}2}(\Ombar\times[T+1,∞))$ 

Moreover, $ξn$ solves the following initial boundary value problem for $\ntilde$: 
\begin{equation}\label{eq:IBVPn}
 \ntilde_t=\Delta \ntilde-a•\na \ntilde +b,\quad \ntilde(T)=0,\quad \delny \ntilde\amrand =0
\end{equation}
where 
\[
 a=\chi \na c + u, \qquad b=-\chi n\Delta c ξ + κnξ-μn^2ξ +ξ'n
\]
satisfy $a,b\in C^{γ,\frac{γ}2}(\Ombar\times(T,\infty))$. \cite[IV.5.3]{LSU}  guarantees the existence of a solution $\ntilde\in C^{2+γ,1+\frac{γ}2}(\Ombar\times[T,\infty))$ for \eqref{eq:IBVPn} and the uniqueness assertion in \cite[III.5.1]{LSU} for weak solutions of \eqref{eq:IBVPn} shows that $ξn=\ntilde$. Due to $ξ\equiv 1$ on $[T+1,∞)$, we conclude $n\in C^{2+γ,1+\frac{γ}2}(\Ombar\times[T+1,\infty))$.

Finally, $ξu$ solves 
\[
 (ξu)_t = Δ(ξu)+ \calP (ξ'u  - ξ(u•∇)u + ξn∇Φ + ξf), \quad \na\cdot(ξu)=0, \quad (ξu)(T-1)=0, \quad (ξu)\bdry=0
\]
where $\calP (ξ'u  - ξ(u•∇)u + ξn∇Φ + ξf)$ is Hölder continuous according to Corollary \ref{cor:convc1pa}, \eqref{reg:f} and \cite[Lemma A.2]{XinruJoh}. 
The regularity assertion of \cite[Thm. 1.1]{Solonnikov}, if combined with the uniqueness result in \cite[Thm. V.1.5.1]{sohr_book}, thus yields the desired smoothness of $ξu$ on $[T-1,∞)$ and hence of $u$ on $[T,∞)$.
\end{proof}

\subsection{Convergence}\label{sec:convergence}

\begin{lemma}\label{lem:convergence}
 The solution $(n,c,u)$ of \eqref{eq:system} constructed in Proposition \ref{prop:ex-weaksol} satisfies 
 \[
  n(•,t)\to \frac{κ}{μ},\qquad c(•,t)\to 0,\qquad u(•,t)\to 0 \qquad\quad  \mbox{ as } t\to \infty 
 \]
 in $C^1(\Ombar)$. 
\end{lemma}
\begin{proof}
Assume $c(t)\not\to 0$ in $C^1(\Ombar)$ as $t\to\infty$. Then there are $η_0>0$ and $t_j\to\infty$ such that $\norm[C^1(\Ombar)]{c(t_j)}>η_0$ for all $j\nat$. Due to \eqref{eq:holderest} and by the compact embedding $C^{1+γ}(\Ombar)\embeddedinto\embeddedinto C^1(\Ombar)$ there is some function $c_∞∈C^1(\Ombar)$ such that $c(t_{j_k})\to c_∞$ along a subsequence $(t_{j_k})_{k\nat}$ of $(t_j)$ and therefore $c(t_{j_k})\to c_∞$ in $\Liom$ as $t\to∞$, which shows that, according to Corollary  \ref{cor:ctozero}, $c_∞=0$. But $\norm[C^1(\Ombar)]{c(t_{j_k})}\to 0$ contradicts $\norm[C^1(\Ombar)]{c(t_j)}>η_0$ for all $j\nat$.

We proceed similarly for $u$: Assuming $u\not\to 0$ in $C^1(\Ombar)$, we find $η_0>0$, $t_j\to\infty$ and $u_∞∈C^1(\Ombar)$ such that $\norm[C^1(\Ombar)]{u(t_{j_k})}>η_0$, $u(t_j)\to u_∞$ in $C^1(\Ombar)$.
If $u_∞\not\equiv 0$, for some arbitrary $p>4$ there are $η_1>0$ and a subsequence $(t_{j_k})_{k\nat}$ of $(t_j)_{j\nat}$ such that $\norm[\Lom p]{u(t_{j_k})}>η_1$.
With $T_0$ as in Corollary \ref{cor:convc1pa}, we may use Lemma \ref{lem:ueptozero} to obtain $T>T_0$ such that for any $t>T$ there is $ε_t>0$ such that for all $ε∈(0,ε_t)$ we have $\norm[\Lom p]{\uep(•,t)}<\frac{η_1}2$, and pick $j$ such that $t_j>T$ and $ε_0$ such that for all $ε∈(0,ε_0)$ $\norm[\Lom p]{\uep(•,t_j)}<\frac{η_1}2$. We observe that thereby $\norm[\Lom p]{u(•,t_j)-\uep(t_j)}>η_1-\frac{η_1}2>0$, which contradicts Corollary \ref{cor:convc1pa}.

As to the convergence of $n$ we define 
\[
 n_j(x,s):=n(x,j+s),\quad x\in\Ombar,\  s\in[0,1]
\]
and claim that $n_j\to \frac{κ}{μ}$ in $C^{1,0}(\Ombar\times[0,1])$ as $j\to\infty$. Were this not the case, we could find $η_0>0$ and a sequence $(k_j)_{j\nat}\subℕ$,  $k_j\to\infty$, such that $\nnorm{C^{1,0}(\Ombar\times[0,1])}{n_{k_j}-\frac{κ}{μ}}>η_0$ for all $j∈ℕ$. Due to the bound on $n$ in \eqref{eq:holderest}, $n_{k_{j_l}}\to n_∞$ in $C^{1,0}(\Ombar\times[0,1])$ with some $n_∞∈C^{1,0}(\Ombar\times[0,1])$. Because $\int_0^1\io \left(n_j(x,s)-\frac{κ}{μ}\right)^2 dxds\to 0$ as $j\to\infty$, according to Lemma \ref{lem:estimatesfromddtG}, $n_∞\equiv \frac{κ}{μ}$, contradicting either $n_{k_{j_l}}\to n_∞$ or $\nnorm{C^{1,0}(\Ombar\times[0,1])}{n_{k_j}-\frac{κ}{μ}}>η_0$. 
Hence $n_j\to\frac{κ}{μ}$ in $C^1(\Ombar\times[0,1])$ as $j\to\infty$, and in particular $\sup_{s\in[0,1]} \norm[C^1(\Ombar)]{n_j(•,s)-\frac{κ}{μ}}\to 0$ as $j\to\infty$ implies that $n(\cdot,t)\to \frac{κ}{μ}$ as $t\to \infty$.
\end{proof}

\subsection{Proof of Theorem \ref{thm:evsmooth}}
\begin{proof}[Proof of Theorem \ref{thm:evsmooth}]
The theorem immediately results from Lemma \ref{lem:c2alpha} and Lemma \ref{lem:convergence}. 
\end{proof}

{\footnotesize

\def\cprime{$'$}

}

\end{document}